\documentclass[10pt]{amsart}
\usepackage{a4}
\usepackage{amssymb}
\usepackage{array}
\usepackage{booktabs}
\usepackage{multirow}
\usepackage{hhline}
\usepackage{graphicx}
\usepackage[usenames,dvipsnames]{color}

\newtheorem{thm}{Theorem}

\newtheorem{cond}{Condition}
\newtheorem{cor}[thm]{Corollary}
\newtheorem{lem}[thm]{Lemma}
\newtheorem{prop}[thm]{Proposition}

\theoremstyle{definition}

\newtheorem{defn}[thm]{Definition}

\newtheorem{rem}[thm]{Remark}

\numberwithin{thm}{section}
\numberwithin{equation}{section}

\newcommand{\norm}[1]{\left\Vert#1\right\Vert}
\newcommand{\abs}[1]{\left\vert#1\right\vert}

\newcommand{\eps}{\varepsilon}

\newcommand{\la}{\langle}
\newcommand{\ra}{\rangle}

\newcommand{\A}{\mathcal{A}}

\newcommand{\B}{\mathcal{B}}

\newcommand{\n}{\mathbb{N}}

\newcommand{\sg} {\sigma}

\newcommand{\z}{\mathbb{Z}}

\newcommand{\om}{\omega}

\newcommand{\prt}{\widehat{\otimes}}
\newcommand{\vnt}{\overline{\otimes}}

\begin{document}
\title[Some Beurling--Fourier algebras on compact groups are operator algebras]{Some Beurling--Fourier algebras on compact groups are operator algebras}

\author{Mahya Ghandehari}
\address{Department of Mathematics and Statistics, Dalhousie University,
Halifax, Nova Scotia, B3H 4R2 Canada }
\email{ghandeh@mathstat.dal.ca}

\author{Hun Hee Lee}
\address{Department of Mathematical Sciences and Research Institute of Mathematics,
Seoul National University, San56-1 Shinrim-dong Kwanak-gu, Seoul 151-747, Republic of Korea}
\email{hunheelee@snu.ac.kr}

\author{Ebrahim Samei}
\address{Department of Mathematics and Statistics, University of Saskatchewan, Saskatoon, Saskatchewan, S7N 5E6, Canada}
\email{samei@math.usask.ca}

\author{Nico Spronk}
\address{Department of Pure Mathematics, University of Waterloo, Waterloo,
Ontario, N2L 3G1, Canada}
\email{nspronk@math.uwaterloo.ca}

\thanks{Hun Hee Lee was supported by  Basic Science
  Research Program through the National Research Foundation of
  Korea(NRF) funded by the Ministry of Education, Science and
  Technology(2012R1A1A2005963). Ebrahim Samei was supported by NSERC under grant no. 366066-09. Nico Spronk
was supported by NSERC under grant no. 312515-05.}



\begin{abstract}
Let $G$ be a compact connected Lie group. The question of when a weighted Fourier algebra on $G$ is completely isomorphic to an operator algebra will be investigated in this paper. We will demonstrate that the dimension of the group plays an important role in the question. More precisely, we will get a positive answer to the question when we consider a polynomial type weight coming from a length function on $G$ with the order of growth strictly bigger than the half of the dimension of the group. The case of $SU(n)$ will be examined, focusing more on the details including negative results. The proof for the positive directions depends on a non-commutative version of Littlewood multiplier theory, which we will develop in this paper, and the negative directions will be taken care of by restricting to a maximal torus.
\end{abstract}

\maketitle

\section{Introduction}

Group algebras $L^1(G)$ for locally compact groups $G$ are some of the most fundamental examples of Banach algebras, which are in some sense far away from $C^*$-algebras, or more generally (non-self-adjoint) operator algebras, i.e. closed subalgebras of $B(H)$ for some Hilbert space $H$. For example, $L^1(G)$ is not Arens regular for infinite group $G$ \cite{Young73} whilst operator algebras are always Arens regular (see \cite[Chapter 4]{DL} for exampe). By endowing an appropriate submultiplicative weight function $\om : G \to [1,\infty)$ the weighted group algebra $L^1(G,\om)$ could be closer to operator algebras in some cases. Indeed, if $G$ is a discrete countable group, then some of weighted group algebras $\ell^1(G,\om)$ actually become Arens regular \cite[Chapter 8]{DL}. Varopoulos proved that even more is true \cite{Var72}. When $G = \mathbb{Z}$ and $\rho_\alpha$ is the polynomial type weight given by
	$$\rho_\alpha(x) =  (1+\abs{x})^\alpha,\;\; x\in \mathbb{Z},\; \alpha\ge 0,$$
$\ell^1(\z, \rho_\alpha)$ is isomorphic to a $Q$-algebra if and only if $\alpha > \frac{1}{2}$. Recall that a $Q$-algebra is a quotient of a uniform algebra, a closed subalgebra of a commutative $C^*$-algebra $C(X)$ for some compact Hausdorff space $X$. Since a quotient algebra of an operator algebra is again an operator algebra (see \cite[Proposition 2.3.4]{BLe04} for example), $Q$-algebras are always operator algebras.

The initial motivation of this paper was to consider a non-commutative version of Varopoulos's result. The correct non-commutative analogue of group algebras are Fourier algebras $A(G)$. The discrete-compact duality suggests us that we might get weighted versions of Fourier algebras on certain compact groups which are isomorphic to operator algebras. Recently, the theory of weighted Fourier algebras have been developed by Ludwig/Spronk/Turowska \cite{LST} and Lee/Samei \cite{LS} under the name of Beurling--Fourier algebras, which we will use as our model of weighted versions of Fourier algebras. Let $G$ be a compact group, and let $\widehat{G}$ be the equivalence classes of irreducible unitary representations of $G$. We call a function $\om : \widehat{G} \to [1,\infty)$ a {\it weight} if
	\begin{equation*}\label{eq-weight1}
	\om(\sigma) \le \om(\pi) \om(\pi')
	\end{equation*}
for any $\pi,\pi'\in \widehat{G}$ and $\sigma\in \widehat{G}$ which appears in the irreducible decomposition of $\pi\otimes \pi'$ (see \cite[section 5.1]{Fol} for the details). In \cite{LST}, the Beurling--Fourier algebra $A(G,\om)$ is defined by
	$$A(G,\om):=
	\{f\in C(G) : \norm{f}_{A(G,\om)}=\sum_{\pi\in \widehat{G}}d_\pi \om(\pi)\norm{\widehat{f}(\pi)}_1 < \infty \},$$
where $\widehat{f}(\pi)$ is the Fourier coefficient of $f$ at $\pi\in \widehat{G}$ and $\norm{\cdot}_1$ denotes the trace norm. Note that the constant weight $\om \equiv 1$ gives us the usual Fourier algebra. In this paper we will be mainly interested in the following weights. The first one is $\om_\alpha$, $\alpha>0$, {\it the dimension weight of order $\alpha$} given by
	$$\om_\alpha(\pi) = d^\alpha_\pi,\;\; \pi\in \widehat{G}.$$
If the compact group $G$ is a connected Lie group, then $\widehat{G}$ is generated by a finite generating set $S$, so that we can consider the associated length function $\tau_S$. In this case we have the second kind of weight $\om^\alpha_S$, which we call {\it the polynomial weight of order $\alpha$}, given by
	$$\om^\alpha_S(\pi) = (1+\tau_S(\pi))^\alpha,\;\; \pi\in \widehat{G}.$$
Note that both of the above weights are of polynomial type. One can also have exponential type weights; the main example is as follows.
For $0\le \alpha \le 1$, we define the {\it exponential weight of order $\alpha$} as
			$$\gamma_\alpha(\pi)=e^{\tau_S(\pi)^\alpha},\;\; \pi\in \widehat{G}.$$
See section \ref{sec-BF-alg-gen} for the details of the above definitions.

This Beurling--Fourier algebra actually first appeared in \cite{J1} under the name of $A_\gamma(G)$, which played an important role in the proof of the fact that $A(SU(2))$ is a non-amenable Banach algebra. The algebra $A_\gamma(G)$ was rediscovered in \cite{FSS1} together with its operator space counterpart $A_\Delta(G)$, where the authors investigated various aspects of amenability of the algebras. In our notations the algebra $A_\gamma(G)$ is nothing but $A(G, \om_1)$. The general concept of Beurling--Fourier algebra has been defined in \cite{LST} and \cite{LS}. The article \cite{LS} focused on amenability properties and Arens regularity of Beurling--Fourier algebras and demonstrated that Beurling--Fourier algebras behave quite differently from the usual Fourier algebra, whilst \cite{LST} focused on the spectral properties showing that Beurling--Fourier algebras capture different aspects of the underlying group compared to the usual Fourier algebras. One important consequence of these results is that by putting them together, one can get a better picture of non-amenability of certain Fourier algebras relating them to the geometric structure of their underlying groups. Johnson's work in \cite{J1} show that the amenability of $A(G)$ related to the behavior of $A_\gamma(G)$, which we know by now it is a weighted $\ell^1$-direct sum of trace classes, whereas by \cite{FSS1}, the operator amenability relates to the behavior of $A_\Delta(G)$, which is a weighted $\ell^1$-direct sum of Hilbert spaces. In fact, it is shown in Theorem \ref{thm-restriction-dim} (see \cite{LS} for $n=2$) that if one restrict $A(SU(n))$ to a maximul torus, then one obtains the weighted group algebra (or the Beurling algebra) $\ell^1(\mathbb{Z}^{n-1}, (1+\|n\|)^{n-1})$ which is known to behave poorly with respect to amenability. These gives us more insightful picture on why amenability fails for Fourier algebras of non-abelian compact connected Lie groups.

As is usual in the theory of Fourier algebras, we will work in the category of operator spaces. This allows us to use D. Blecher's completely isomorphic characterization of operator algebras requiring the algebra multiplication map to be completely bounded on the Haagerup tensor product \cite{B95}. Note that there is no such characterization of operator algebras in the category of Banach spaces \cite{Car}.

We summarize the main results of this paper. It turns out that there is an interesting connection between $d(G)$, the dimension of the Lie group $G$ and the property of the associated Beurling--Fourier algebra being completely isomorphic to an operator algebra. Proofs of these results will be given in Section \ref{sec-BF-OP}.

	\begin{thm}
	Let $G$ be a compact connected Lie group. Then $A(G, \om^\alpha_S)$ is completely isomorphic to an operator algebra if $\alpha>\frac{d(G)}{2}$ and fails to be completely isomorphic to an operator algebra if
$G=SU(n)$, $n\ge 2$ and $\alpha \leq \frac{n-1}{2}$. Moreover, $A(G,\gamma_\alpha)$ on $G$ with respect to the exponential weight of order $\alpha$ is completely isomorphic to an operator algebra if $0< \alpha < 1$.
	\end{thm}

The situation for the dimension weights is more delicate. First we show that, if $G$ is not simple
(as a compact, connected Lie group), then one can not get an operator algebra as an isomorphic image of a Beurling--Fourier algebras on $G$ coming from a dimension weight (see Theorem \ref{T:dim weight-non simple-non operator alg}). Hence we need to restrict to compact, connected {\it simple} Lie groups to achieve positive results.
However, even though we have developed the general theory, the computations become extremely technical
as the dimension of the Lie group grows even for the most classical case of compact simple Lie group, namely $SU(n)$. Nonetheless, we have the following results which give some evidence that one may obtain various classes of Beurling--Fourier
algebras coming from dimension weights which are completely isomorphic to operator algebras.
\begin{thm}
	The Beurling--Fourier alegrba $A(SU(n), \om_\alpha)$, $n\ge 2$ is completely isomorphic to an operator algebra if $\alpha>\frac{d(SU(n))}{2} = \frac{n^2-1}{2}$ and fails to be completely isomorphic to an operator algebra if $\alpha \leq \frac{1}{2}$.
	\end{thm}

 It is natural to ask whether the exponent $\frac{d(G)}{2}$, obtained in the preceding theorems, is optimal. We could demonstrate the optimal exponent of $\frac{d(G)}{2}$ only for the $n$-dimensional torus (Theorem \ref{thm-torus}). On the other hand, the values of $\alpha$ for which we obtain negative results for $SU(n)$ are quite smaller than $\frac{d(G)}{2}$ and we are not aware of any means to improve this gap.

We would like to point out that we can not hope the Beurling--Fourier algebra on $SU(n)$ to be completely isomorphic to a $Q$-algebra since it may not be even completely isomorphic to a $Q$-space, a quotient operator space of a closed subspace of a commutative $C^*$-algebra (see Remark \ref{R:BF SU(n)-not Q alg}).

This paper is organized as follows. In section \ref{sec-Littlewood}, we develop a non-commutative Littlewood multiplier theory, which is a main tool for the proof of positive results. This requires a heavy use of operator spaces, so that we collect the necessary back ground materials on operator spaces and operator algebras in the beginning of the section. Section \ref{sec-BF-alg-gen} starts with a brief introduction of Beurling--Fourier algebras on compact groups and dimension weights. Then the definition of polynomial weights on connected compact Lie groups will follow after some preliminaries of corresponding Lie theory. We will close the section with a more detailed representation theory of $SU(n)$ and restriction results of weights to a maximal torus. In section \ref{sec-BF-OP}, we present our mains results starting with a complete solution of the problem in the case of $n$-dimensional torus, and then we focus on polynomial type weights
and dimension weights. We will also prove positive results for exponential type of weights. We will study in details the case of $SU(n)$ as our main example of a compact connected, simple Lie group.
In the appendix we present two technical proofs concerning estimates of the dimension weight and the exponential weight.

We finish this section by pointing out an interesting and, at the same time, surprising contrast between Beurling algebras (i.e. weighted group algebras) and Beurling--Fourier algebras. It is shown in \cite{FSS1} that Beurling--Fourier algebras on compact groups with certain dimension weights can be embedded as a closed subalgebra of other Fourier algebras. An important consequence of such fact is that one can construct infinite-dimensional closed subalgebra of Fourier algebras on products of $SU(n)$ which are completely isomorphic to operator algebras (Corollary \ref{C:op. subalgebra of Fourier alg}). We are not aware of an analogous result in the commutative case (i.e. for Beurling algebras), so that this seems to be a non-commutative phenomenon.

\section{Some non-commutative Littlewood multipliers}\label{sec-Littlewood}

\subsection{Preliminaries on operator spaces and operator algebras}
We will assume that the reader is familiar with standard operator space theory including injective, projective and Haggerup tensor products of operator spaces. However, in this section we will recall some operator space theory which is somewhat less standard and will be used frequently later on.

The column and the row Hilbert spaces on a Hilbert space $H$ will be denoted by $H_c$ and $H_r$. Note that they are given by
	$$H_c = B(\mathbb{C}, H)\;\;\text{and}\;\; H_r = B(\overline{H},\mathbb{C}),$$
where $\overline{H}$ is the complex conjugate of $H$. When dim$H = n<\infty$, then $H_c$ and $H_r$ are usually denoted by $C_n$ and $R_n$.

For any operator space $E\subseteq B(H)$ and $T\in CB(C_n, E)$ we have the following concrete formula to calculate the completely bounded norm (shortly, cb-norm) of $T$.
	\begin{equation}\label{eq-cb-norm}
	\norm{T}_{cb} = \norm{\sum^n_{i=1} Te_i(Te_i)^*}^{\frac{1}{2}}_{B(H)},
	\end{equation}
where $\{e_i\}^n_{i=1}$ is an orthonormal basis of $C_n$. A similar formula for $T\in CB(R_n, E)$ is also available.
For operator spaces $(F_i)_{i\in I}$ we have
	\begin{equation}\label{eq-directsum}
	CB(E, \oplus_{i\in I}F_i) \cong \bigoplus_{i\in I} CB(E,F_i)
	\end{equation}
completely isometrically via the following identification.
	$$T \mapsto \oplus_{i\in I} (P^F_i \circ T),$$
where $P^F_j : \bigoplus_{i\in I} F_i \rightarrow F_j$ is the canonical projection, which is a complete contraction.

The column and the row Hilbert spaces are closely related to the Haggerup tensor product. In this paper we will more concerned about its dual version, namely the extended Haggerup tensor product. The extended Haggerup tensor product of dual operator spaces $E^*$ and $F^*$ will be denoted by
	$$E^*\otimes_{eh} F^*$$
which is given by $(E\otimes_h F)^*$ (\cite{BS}). Then we have the following complete isometry \cite[Lemma 13.3.1]{ER}
	$$E^*\otimes_{eh} F^* \cong \Gamma^c_2(F, E^*)$$
via the map
	$$A \otimes B \mapsto u,\;\; \text{where}\;\; u(X) = \la X, B\ra A.$$
Here $\Gamma^c_2(E,F)$ is the space of mappings factorized through column Hilbert space \cite[Section 13.3]{ER}. More precisely, $u \in \Gamma^c_2(E,F)$ if and only if there is a Hilbert space $H$ and $A \in CB(E, H_c), B\in CB(H_c, F)$ such that
	$$u = B \circ A.$$
The space $\Gamma^c_2(E,F)$ is equipped with the norm $\gamma^c_2(u) = \inf \norm{A}_{cb}\norm{B}_{cb}$, where the infimum runs over all possible such factorization. We have a natural operator space structure on $\Gamma^c_2(E,F)$ given by the matricial norms
	$$\gamma^c_{2,n}((u_{ij})) = \inf \norm{A}_{cb}\norm{B}_{cb},\; (u_{ij}) \in M_n(\Gamma^c_2(E,F))$$
where the infimum runs over all possible
	$$A \in CB(E, M_{1,n}(H_c))\;\; \text{ and }\;\; B \in CB(H_c, M_{n,1}(F))$$
with
	$$(u_{ij}) = B_{1,n} \circ A.$$
Recall that for any linear map $T : E\to F$ between operator spaces we denote the amplified map
	$$id_{M_{m,n}}\otimes T : M_{m,n}(E) \to M_{m,n}(F)$$
simply by $T_{m,n}$.

There is one more tensor product we will use frequently later on, namely the normal spatial tensor product. For any two dual operator spaces $E^*$ and $F^*$ we define the normal spatial tensor product $E^*\overline{\otimes}F^*$ by the weak$^*$-closure of the algebraic tensor product $E^*\otimes F^*$ in $CB(F, E^*)$ via the same identification map as above. Note that 		
	$$E^*\overline{\otimes}F^* = CB(F, E^*)$$
holds if $E$ satisfies the operator space approximation property (shortly OAP). See \cite[chapter 11]{ER} for the details.
	
We close this section with some operator algebra related notations. Let $\A$ be an operator algebra. We say that an operator space $E \subseteq B(H)$ is an {\it (abstract) operator left $\A$-module} if $E$ is a left $\A$-module with the left $\A$-module map
	$$\varphi : \A \times E \to E$$
and there are a complete isometry $j_E : E \rightarrow B(K)$ and a completely contractive map $j_\A : \A \rightarrow B(K)$ for some Hilbert space $K$ satisfying
	$$j_\A(X) j_E(Y) = j_E(\varphi(X,Y)),\;\; X\in \A, Y\in E.$$
We also say that $E$ is a left $h$-module over $\A$ if the module map $\varphi$ (understood as the associated linear map) extends to a completely bounded map
	$$\varphi : \A \otimes_h E \to E.$$
It is straightforward to check that any left operator $\A$-module is an left $h$-module over $\A$. Note that operator right $\A$-modules and right $h$-modules are similarly defined. Note also, in passing, that under mild assumptions, $h$-modules and operator modules are naturally isomorphic. See \cite[Theorem 3.3.1]{BLe04} for example.

\subsection{Non-commutative Littlewood multipliers}		
Let $\Sigma$ be a set for which we have a prescribed collection of natural numbers $(d_\sigma)_{\sigma\in\Sigma}$. If $1\leq p<\infty$ and $d\in \n$ we let $S^p_d$ denote $M_d$ equipped with the Schatten $p$-norm $\norm{\cdot}_p$.

We consider subspaces of the product space $\prod_{\sigma\in\Sigma}M_{d_\sigma}$
with associated norms for their elements:
\begin{align*}
L^\infty&=L^\infty(\Sigma)=\ell^\infty\text{-}\bigoplus_{\sigma\in\Sigma}M_{d_\sigma},
&\norm{A}_{L^\infty}&=\sup_{\sigma\in\Sigma}\norm{A_\sigma}_\infty \\
L^2&=L^2(\Sigma)=\ell^2\text{-}\bigoplus_{\sigma\in\Sigma}\sqrt{d_\sigma} S^2_{d_\sigma},\quad
&\norm{A}_{L^2}&=\left[\sum_{\sigma\in\Sigma}d_\sigma{\norm{A_\sigma}^2_2}\right]^{\frac{1}{2}} \\
L^1&=L^1(\Sigma)=\ell^1\text{-}\bigoplus_{\sigma\in\Sigma}d_\sigma S^1_{d_\sigma},
&\norm{A}_{L^1}&=\sum_{\sigma\in\Sigma}d_\sigma\norm{A_\sigma}_1.
\end{align*}

The space $L^1$ is the predual of the space $L^\infty$ via the following {\it standard} duality bracket.
	$$\la (A_\sigma), (B_{\sigma'}) \ra = \sum_{\sigma\in \Sigma} \text{Tr}(A_\sigma B_\sigma),\;\; (A_\sigma) \in L^1, (B_{\sigma'})\in L^\infty.$$
Whenever we consider $L^\infty$ and $L^1$ as operator spaces, we assume their natural operator space structure
as a von Neumann algebra and the predual of a von Neumann algebra, respectively.

	\begin{prop}\label{prop-L2-embed}
	The formal identities ${\rm id}^c_{2,\infty} : L^2_c \to L^\infty$, ${\rm id}^r_{2,\infty} : L^2_r \to L^\infty$,
	${\rm id}^c_{1,2} : L^1 \to L^2_c$ and ${\rm id}^r_{1,2} : L^1 \to L^2_r$ are complete contractions.
	\end{prop}
\begin{proof}
We will only check the case of ${\rm id}^c_{2,\infty}$ since the case ${\rm id}^r_{2,\infty}$ is similar. Moreover, $({\rm id}^c_{1,2})^* = {\rm id}^r_{2,\infty}$ and
$({\rm id}^r_{1,2})^* = {\rm id}^c_{2,\infty}$.

Since
	$$L^2_c = (\ell^2\text{-}\bigoplus_{\sigma\in \Sigma}\sqrt{d_\sigma}S^2_{d_\sigma})_c,$$
it is enough to show that the formal identity ${\rm id}_n : (\sqrt{n}S^2_{n})_c \rightarrow M_{n}$
is a complete contraction for any $n\ge 1$.
Indeed, by \eqref{eq-directsum} we have
	$$\norm{{\rm id}^c_{2,\infty}}_{cb} = \sup_{\sigma\in \Sigma} \norm{Q_\sigma \circ {\rm id}^c_{2,\infty}}_{cb},$$
where
	$$Q_\sigma : \ell^\infty\text{-}\bigoplus_{\rho \in \Sigma} M_{d_\rho} \rightarrow M_{d_\sigma}$$
is the canonical projection, which is completely contractive. Moreover, we have
	$$Q_\sigma \circ {\rm id}^c_{2,\infty} = {\rm id}_\sigma \circ P_\sigma,\; \sigma\in \Sigma$$
where ${\rm id}_\sigma : (\sqrt{d_\sigma}S^2_{d_\sigma})_c \rightarrow M_{d_\sigma}$ is the formal identity and
	$$P_\sigma : \ell^2\text{-}\bigoplus_{\rho \in \Sigma} \sqrt{d_\rho} S^2_{d_\rho} \rightarrow \sqrt{d_\sigma} S^2_{d_\sigma}$$
is the canonical orthogonal projection, which explains that ${\rm id}^c_{2,\infty}$ is completely contractive.

For the claim itself we let $\{e_{ij}\}^{n}_{i,j=1}$ be the matrix units in $M_{n}$. Since
	$$\norm{e_{ij}}_{\sqrt{n}S^2_{n}} = \sqrt{n},$$
\eqref{eq-cb-norm} tells us that
	$$\norm{{\rm id}_n}_{cb} = \norm{\sum^{n}_{i,j=1} n^{-\frac{1}{2}}e_{ij} (n^{-\frac{1}{2}}e_{ij})^*}^{\frac{1}{2}}_{M_{n}}
	= \norm{\sum^{n}_{i=1}e_{ii}}_{M_{n}}^\frac{1}{2} = 1.$$
\end{proof}

We need to understand the $L^\infty$-module structure on $L^2$ as follows.

	\begin{prop}\label{prop-L-inf-mod}
	The space $L^2_c$ is a left operator $L^\infty$-module under the multiplication
		$$AB = (A_\sigma B_\sigma)_{\sigma\in \Sigma}$$
	with $A= (A_\sigma )_{\sigma\in \Sigma} \in L^\infty$ and $B= (B_\sigma )_{\sigma\in \Sigma} \in L^2$, and $L^2_r$ is a right operator $L^\infty$-module under a similar multiplication.
	\end{prop}
\begin{proof}
We only check the case of $L^2_c$.
Note that $L^2_c = B(\mathbb{C}, L^2)$ can be completely isometrically embedded in $B(L^2)$ by the embedding
	$$j_2 : L^2_c \hookrightarrow B(L^2), \;\; Z\mapsto N_Z,$$
where $N_Z(A) = \la A, \psi \ra Z$ for some fixed unit vector $\psi \in L^2$.
If we consider the following standard representation of $L^\infty$
	$$j_\infty : L^\infty \hookrightarrow B(L^2),\;\; X\mapsto M_X,$$
where $M_X$ is the left multiplication by $X$. Then, we have
	$$j_\infty(X)j_2(Z) = j_2(XZ)$$
for any $X\in L^\infty$ and $Z \in L^2$. Thus, $L^2_c$ is a left operator $L^\infty$-module.

\end{proof}

The above module structure can be easily extended to vector-valued cases.

	\begin{lem}\label{lem-dualopaction}
	Let $M$ and $N$ be von Neumann algebras and $E$ a dual left operator $M$-module, i.e. a dual operator space which is a left operator $M$-module and whose
predual is a right operator $M$-module. Then $E\overline{\otimes}N$ is a left operator $M\overline{\otimes}N$-module.
	\end{lem}
\begin{proof}
By \cite[Theorem 3.8.3]{BLe04} we may assume that there is a Hilbert space $H$ for which $M, E\subset B(H)$ as weak$^*$-closed subspaces, and $E$ is a left $M$-module. We let $x\in M\overline{\otimes}N$ and $y\in E\overline{\otimes}N$ and find bounded nets $(x_\alpha)\subset M\otimes N$ and $(y_\beta)\subset E\otimes N$ which converge to $x$ and $y$, respectively, in the weak$^*$-topology.  Then $x_\alpha y_\beta\in E\overline{\otimes} N$ and converges (with limit taken in either order) to $xy$ in $B(H)\overline{\otimes}N$ in the weak$^*$-topology, and hence in the closed subspace $E\overline{\otimes} N$.
\end{proof}

Now we define non-commutative Littlewood multiplier spaces.
	\begin{defn}
	We define the spaces $\mathcal{T}^2_c$ and $\mathcal{T}^2_r$ by
		$$\mathcal{T}^2_c=\mathcal{T}^2_c(\Sigma)=L^2_c\overline{\otimes}L^\infty\;\; \text{and}\;\;
		\mathcal{T}^2_r=\mathcal{T}^2_r(\Sigma)=L^\infty\overline{\otimes}L^2_r.$$
	\end{defn}

\begin{rem}{\rm
In the classical case, namely when $d_\sigma = 1$ for all $\sigma$, we have $L^p=\ell^p=\ell^p(\Sigma)$ for $p=1,2,\infty$. It is straightforward to compute, via the identifications $\mathcal{T}^2_r\cong CB(\overline{\ell^2}_c,\ell^\infty)=B(\overline{\ell^2},\ell^\infty)$, and $\mathcal{T}^2_c\cong CB(\ell^1,\ell^2_c)=B(\ell^1,\ell^2)$ that
\begin{align*}
\mathcal{T}^2_r=\left\{[a_{\sigma,\tau}]:\sup_\sigma \sum_\tau |a_{\sigma,\tau}|^2<\infty\right\},\;
\mathcal{T}^2_c=\left\{[b_{\sigma,\tau}]:\sup_\tau \sum_\sigma |b_{\sigma,\tau}|^2<\infty\right\}.
\end{align*}
Hence we recover the classical Littlewood function space $\mathcal{T}^2$ (\cite{BF}) by the sum of the above two space $\mathcal{T}^2_r+\mathcal{T}^2_c$. Note that the non-commutative setting above forces us to consider left (row) and right (column) cases separately. Note that we are using the term ``Littlewood multipliers" instead of ``Littlewood functions" as is used in the literature, which suits better in non-commutative contexts.
}
\end{rem}

	\begin{cor}\label{cor-moreopmod}
	The space $\mathcal{T}^2_c$ is a left, and $\mathcal{T}^2_r$ is a right, operator $L^\infty\overline{\otimes}L^\infty$-module.
	\end{cor}
\begin{proof} We note that $L^2_c$ is reflexive, and thus a dual $L^\infty$-module, and is a left operator
$L^\infty$-module from Proposition \ref{prop-L-inf-mod}.
We thus apply Lemma \ref{lem-dualopaction} directly to see the result for $\mathcal{T}^2_c$. The result for $\mathcal{T}^2_r$ follows similarly.
\end{proof}

\begin{prop}\label{prop-embedding1}
	The following formal identities are complete contractions.
		$$I_c : \mathcal{T}^2_c \rightarrow L^\infty \otimes_{eh} L^\infty \;\;\text{and}\;\; I_r : \mathcal{T}^2_r \rightarrow L^\infty \otimes_{eh} L^\infty$$
	\end{prop}
\begin{proof}
We again check the case of $I_c$ only. The other one follows similarly.

Note that $\mathcal{T}^2_c = L^2_c\overline{\otimes}L^\infty$ and $L^\infty \otimes_{eh} L^\infty$ can be identified with
$CB(L^1, L^2_c)$ and $\Gamma^c_2(L^1, L^\infty)$ under the (essentially) same identification
	$$A \otimes B \mapsto u,\;\; \text{where}\;\; u(X) = \la X, B\ra A.$$
Thus it is enough to show that the map
	$$CB(L^1, L^2_c) \to \Gamma^c_2(L^1, L^\infty),\;\; X \mapsto {\rm id}^c_{2,\infty}\circ X$$
is a complete contraction, where ${\rm id}^c_{2,\infty}: L_c \to L^\infty$ is the formal identity.
We start with an element
	$$(u_{ij}) \in M_n(CB(L^1, L^2_c)),$$
where $u_{ij} : L^1 \rightarrow L^2_c$. We set
	$$U : L^1 \rightarrow M_{1,n}(M_{n,1}(L^2_c)),\;\; x \mapsto [(u_{ij}(x))^n_{i=1}]^n_{j=1}$$
and
	$$V = {\rm id}_{M_{n,1}}\otimes {\rm id}^c_{2,\infty} : M_{n,1}(L^2_c) \rightarrow M_{n,1}(L^\infty).$$
Since
	$$M_{1,n}(M_{n,1}(L^2_c)) \cong M_n(L^2_c)$$
naturally we have $\norm{U}_{cb} = \norm{(u_{ij})}_{cb}$, and clearly by Proposition \ref{prop-L2-embed}, $\norm{V}_{cb} = 1$.
Moreover, we have
	$$(u_{ij}) = V_{1,n} \circ U$$
and $M_{n,1}(L^2_c) \cong C_n \otimes_h L^2_c$ is also a column Hilbert space.
Thus
$$\gamma^c_{2,n}(u_{ij}) \leq \|V\|_{cb}\|U\|_{cb}=\|(u_{ij})\|_{M_n(CB(L^1, L^2_c))}.$$
This completes the proof.

\end{proof}

Combining Corollary \ref{cor-moreopmod} and Proposition \ref{prop-embedding1} we get the following
	\begin{thm}\label{thm-multiplier}
	Every element in $\mathcal{T}^2_c$ (resp. $\mathcal{T}^2_r$) is a right (resp. left)
	cb-multiplier from $L^\infty \overline{\otimes} L^\infty$ into $L^\infty \otimes_{eh} L^\infty$ with the same cb-norm.
	\end{thm}

\section{Beurling--Fourier algebras on compact groups}\label{sec-BF-alg-gen}

In this section we collect basic materials concerning Beurling--Fourier algebras on compact groups.

\subsection{Preliminaries} Let $G$ be a compact group. We will use the notation
	$$\sigma \subset \pi\otimes \pi',\; \pi,\pi'\in \widehat{G}$$
which implies that $\sigma \in \widehat{G}$ appears in the irreducible decomposition of $\pi\otimes \pi'$.

The group von Neumann algebra $VN(G)$ of G is defined by
	$$VN(G) = \{\lambda(x): x\in G\}'' \subset B(L^2(G)),$$
where $\lambda$ is the left regular representation of $G$.
$VN(G)$ is equipped with the co-multiplication
	$$\Gamma : VN(G) \to VN(G)\vnt VN(G),\;\; \lambda(x) \mapsto \lambda(x)\otimes \lambda(x).$$
Using representation theory of $G$ we have an equivalent formulation of $VN(G)$, namely
	$$VN(G) \cong \ell^\infty\text{-}\bigoplus_{\pi\in \widehat{G}} M_{d_\pi}$$
under the $*$-isomorphism
	$$\lambda(x) \mapsto (\bar{\pi}(x))_{\pi\in \widehat{G}},\; x\in G.$$
We note that $VN(G)$ acting on $L^2(G)$ is unitarily equivalent to $\ell^\infty\text{-}\bigoplus_{\pi\in \widehat{G}} M_{d_\pi}$ as a von Neumann algebra acting on $\ell^2\text{-}\bigoplus_{\pi\in \widehat{G}} \sqrt{d_\pi}S^2_{d_\pi}$ by left multiplication.

We will frequently use the above identification without further comment. For example, we will understand $(A(\pi))_\pi \in \ell^\infty\text{-}\bigoplus_{\pi\in \widehat{G}} M_{d_\pi}$ as an element of $VN(G)$. By abuse of notation we will denote $\Gamma$ transferred to
	$$\ell^\infty\text{-}\bigoplus_{\pi\in \widehat{G}} M_{d_\pi} \to
	\Big(\ell^\infty\text{-}\bigoplus_{\pi\in \widehat{G}} M_{d_\pi}\Big) \vnt
	\Big(\ell^\infty\text{-}\bigoplus_{\pi\in \widehat{G}} M_{d_\pi}\Big)$$
again by $\Gamma$. The formula for the transferred one is the following, which is a folklore,
but we include the proof for the convenience of the readers.
	\begin{prop}\label{prop-comulti-formula}
	For any $A = (A(\pi))_{\pi\in \widehat{G}}$ we have
		$$\Gamma(A) = (X(\pi, \pi'))_{\pi,\pi'\in \widehat{G}}\;\; \text{with}\;\;
		X(\pi,\pi') \cong \bigoplus_{\sigma \subset \pi\otimes \pi'}A(\sigma)$$
	up to unitary equivalences.
	\end{prop}
\begin{proof}
It is straightforward to check that the formula holds for $\lambda(x) \cong (\bar{\pi}(x))_{\pi\in \widehat{G}}$ for any $x\in G$.
Now we apply weak$^*$-density of the linear span of $\{\lambda(x):x\in G\}$ in $VN(G)$ to get the result in full generality.
\end{proof}

\subsection{Beurling--Fourier algebras}

We refer the reader to \cite{LS, LST} for the details of this section.

Let $G$ be a compact group. We call a function $\om : \widehat{G} \to [1,\infty)$ a {\it weight} if
	\begin{equation}\label{eq-weight1}
	\om(\sigma) \le \om(\pi) \om(\pi')
	\end{equation}
for any $\pi,\pi'\in \widehat{G}$ and $\sigma\in \widehat{G}$ satisfying $\sigma \subset \pi \otimes \pi'$ (see \cite[Theorem 2.12]{LS} or
\cite[Section 3]{LST}). For any weight $\om$, we set
	$$W = \bigoplus_{\pi\in \widehat{G}} \om(\pi){\rm id}_{M_{d_\pi}}.$$
Note that $W$ is an unbounded operator in general, but $\Gamma(W)$ still can be well-defined also as an unbounded operator (\cite[Section 2 and Theorem 2.12]{LS}).
We may view $\Gamma(W)$ as a collection of matrices with possibly unbounded matrix norms.
Moreover,
	$$W^{-1} = \bigoplus_{\pi\in \widehat{G}} \om(\pi)^{-1} {\rm id}_{M_{d_\pi}}$$
is a bounded operator. Then, Proposition \ref{prop-comulti-formula} tells us that
	\begin{equation}\label{eq-weight3}
	\Gamma(W)(W^{-1}\otimes W^{-1}) (\pi, \pi') \cong \bigoplus_{\sigma \subset \pi\otimes \pi'}\frac{\om(\sigma)}{\om(\pi)\om(\pi')}{\rm id}_{M_{d_\sigma}},
	\end{equation}
so that the condition \eqref{eq-weight1} can be restated as
	\begin{equation}\label{eq-weight2}
	\Gamma(W)(W^{-1}\otimes W^{-1}) \le 1_{VN(G)}.
	\end{equation}
We define the {\it Beurling--Fourier algebra} $A(G,\om)$ by
	$$A(G,\om):=\{f\in C(G) : \norm{f}_{A(G,\om)}=\sum_{\pi\in \widehat{G}}d_\pi \om(\pi)\norm{\widehat{f}(\pi)}_1 < \infty \},$$
where
	$$\widehat{f}(\pi) = \int_G f(x)\overline{\pi}(x)dx \in M_{d_\pi}.$$
$A(G,\om)$ can be naturally identified with the space
	$$\ell^1\text{-}\bigoplus_{\pi\in \widehat{G}} d_\pi \om(\pi)S^1_{d_\pi}.$$
Thus the dual space is
	$$\ell^\infty\text{-}\bigoplus_{\pi\in \widehat{G}} \om(\pi)^{-1}M_{d_\pi}$$
via the standard duality bracket, which we will denote by $VN(G,W^{-1})$.
The above notation is justified by the fact that $VN(G,W^{-1})$ can be identified with $\{AW : A\in VN(G)\}$ endowed with the norm
	$$\norm{AW}_{VN(G,W^{-1})} = \norm{A}_{VN(G)}.$$
Thus, we have a canonical isometry $VN(G) \rightarrow VN(G,W^{-1}),\; A \mapsto AW$,
and we equipped an operator space structure on $VN(G,W^{-1})$ using this isometry.
Consequently, $A(G,\om)$ has a natural operator space structure as the predual of $VN(G,W^{-1})$,
with which $A(G,\om)$ is a completely contractive Banach algebra under the pointwise multiplication.
Indeed, the cb-norm of the pointwise multiplication
	$$m : A(G,\om) \prt A(G,\om) \to A(G,\om)$$
is known to be the same as the cb-norm of the modified co-multiplication
	\begin{equation}\label{eq-mod-co-multi}
	\widetilde{\Gamma} : VN(G) \rightarrow VN(G)\vnt VN(G),\; A \mapsto \Gamma(A)\Gamma(W)(W^{-1}\otimes W^{-1}),
	\end{equation}
and the condition \eqref{eq-weight1} or \eqref{eq-weight2} implies that $\widetilde{\Gamma}$ is completely contractive.

We finish this section with the definition of a fundamental example of weights on $\widehat{G}$.
	\begin{defn}\label{def-dimension-weight}
		For $\alpha \geq 0$, we define $\om_\alpha : \widehat{G} \to [1,\infty)$ by
			$$\om_\alpha(\pi) = d^\alpha_\pi \;\; \ \ (\pi\in \widehat{G}).$$
Clearly $\om_\alpha$ satisfies the condition \eqref{eq-weight1}, and so, it defines a weight on $\widehat{G}$; it is called the {\it dimension weight of order $\alpha$}.
	\end{defn}
 We would like to point out that if $G$ is abelian,
then $\om_\alpha=1$. Hence the dimension weights are interesting only for compact groups that are far
from being abelian.

\subsection{Weights on the dual of compact connected Lie groups}

When the group $G$ is a connected compact Lie group,
we have another fundamental example of weights on $\widehat{G}$ using the highest weight theory.
See \cite{Wal73} or \cite[section 5]{LST} for the details.

Let $\mathfrak{g}$ be the Lie algebra of $G$ with the decomposition $\mathfrak{g} = \mathfrak{z} + \mathfrak{g}_1$,
where $\mathfrak{z}$ is the center of $\mathfrak{g}$ and $\mathfrak{g}_1 = [\mathfrak{g}, \mathfrak{g}]$.
Let $\mathfrak{t}$ be a maximal abelian subalgebra of $\mathfrak{g}_1$ and $T = \text{exp}\mathfrak{t}$.
Then there are fundamental weights $\lambda_1, \cdots, \lambda_r, \Lambda_1,\cdots,\Lambda_l \in \mathfrak{g}^*$ with $r = \text{dim}\mathfrak{z}$ and $l=\text{dim}\mathfrak{t}$ such that
any $\pi \in \widehat{G}$ is in one-to-one correspondence with its associated highest weight
\begin{equation}\label{eq-highest weight rep}
\Lambda_\pi = \sum^r_{i=1}a_i\lambda_i + \sum^l_{j=1}b_j \Lambda_j,
\end{equation}
which is parameterized by $r$ integers $(a_i)^r_{i=1}$ and $l$ non-negative integers $(b_j)^l_{j=1}\in \z^l_+$.
Note that we adapted the same notations for the weights $\lambda_i$ and $\Lambda_j$ from \cite[section 5]{LST}.

Let $\chi_i$ be the character of $G$ associated to the highest weight $\lambda_i$ and $\pi_j$ be the irreducible representation associated to the weight $\Lambda_j$. Then,
	$$S=\{\pm \chi_i, \pi_j : 1\le i\le r, 1\le j \le l\}$$
is known to generate $\widehat{G}$. More precisely, if we denote for every $k\ge 1$,
	$$S^{\otimes k} =\{\pi \in \widehat{G} : \pi \subset \sigma_1 \otimes \cdots \otimes \sigma_k \;\;
	\text{where}\;\; \sigma_1, \cdots, \sigma_k \in S\cup\{1\} \},$$
then we have
	$$\bigcup_{k\ge 1} S^{\otimes k} = \widehat{G}.$$
	Now we define $\tau_S : \widehat{G} \to \mathbb{N}\cup\{0\}$, {\it the length function on $\widehat{G}$ associated to $S$}, by
	$$\tau_S(\pi) := k,\;\;\text{if}\;\; \pi\in S^{\otimes k}\backslash S^{\otimes(k-1)}.$$

From the definition, we clearly have
	\begin{equation}\label{eq3}
	\tau_S(\sigma) \le \tau_S(\pi) + \tau_S(\pi')
	\end{equation}
for any $\pi, \pi'\in \widehat{G}$ and $\sigma \subset \pi\otimes \pi'$. This fact allows us
to use $\tau_S$ to construct various weights on $\widehat{G}$.
	\begin{defn}\label{def-poly-weight}
		For $\alpha \geq 0$ and $1\geq \beta \geq 0$, we define $\om^\alpha_S, \gamma^\beta_S : \widehat{G} \to [1,\infty)$ by
			$$\om^\alpha_S(\pi) = (1+\tau_S(\pi))^\alpha \ , \ \gamma^\beta_S(\pi)=e^{\tau_S(\pi)^\beta}\ \ (\pi\in \widehat{G}).$$
	Using \eqref{eq3}, it follows routinely that both $\om^\alpha_S$ and $\gamma^\beta_S$ satisfy \eqref{eq-weight1}, and hence, they define weights on $\widehat{G}$; they are called the {\it polynomial weight of order $\alpha$} and the {\it exponential weight of order $\beta$}, respectively. When
$G$ is abelian (e.g. $G=\mathbb{T}^n$), then our definitions coincide with the classical polynomial
and exponential weights on finitely generated abelian groups.
		\end{defn}
		
\begin{rem}\label{re-lenght-norm 1-equ}
	\begin{enumerate}
	\item
We would like to highlight the fact that the above length function $\tau_S$ is equivalent to the following 1-norm defined on $\widehat{G}$:
	$$\norm{\pi}_1 := \sum^r_{i=1}\abs{a_i} + \sum^l_{j=1}b_j,$$
where the integers $a_i$ and $b_j$ are defined in \eqref{eq-highest weight rep}.
Indeed, in the proof of \cite[Theorem 5.4]{LST}, it is proved that there is a constant $C$ depending only on $G$ such that
\begin{equation}\label{eq2}
	\tau_S(\pi) \le \norm{\pi}_1 \le C \tau_S(\pi).
	\end{equation}
	
	\item
	We may consider a variant of exponential weights of the form $e^{D\tau_S(\pi)^\beta}$ with an additional parameter $D>0$. We note that all the results in this paper concerning the weight $e^{\tau_S(\pi)^\beta}$ still hold in the case of the weight $e^{D\tau_S(\pi)^\beta}$ with a minor modification of calculations.
	\end{enumerate}
	
	\end{rem}

\subsection{Weights on the dual of $SU(n)$ and its restriction to a maximal torus}

The classical group $SU(n)$ is semisimple, so that we have $\mathfrak{z} = 0$.
We denote the maximal torus of $SU(n)$ consisting of diagonal matrices by $H_n \cong \mathbb{T}^{n-1}$.
Then $\widehat{SU(n)}$ is in one-to-one correspondence with $(n-1)$-tuples
	$$(b_1\cdots, b_{n-1})\in \z^{n-1}_+.$$
Note that the canonical generating set is given by
	$$S = \{(1, 0, \cdots, 0), \cdots, (0, \cdots, 0, 1) \}\subset \z^{n-1}_+.$$
By setting
	$$\lambda_n = 0, \lambda_{n-1} = b_{n-1}, \lambda_{n-2} = b_{n-2} + b_{n-1}, \cdots, \lambda_1 = b_1 + \cdots + b_{n-1},$$ we get a one-to-one correspondence between $\widehat{SU(n)}$ and $n$-tuples $\lambda=(\lambda_1\cdots, \lambda_n)\in \z^n_+$ satisfying
	$$\lambda_1 \ge \lambda_2 \ge \cdots \ge \lambda_{n-1} \ge \lambda_n=0.$$
We will denote the $n$-tuple by $\lambda = (\lambda_1\cdots, \lambda_n)$, which is usually called a {\it dominant weight} in Lie theory. See \cite{FH} for the details of representation theory of $SU(n)$.
Let $\pi_\lambda$ be the irreducible representation corresponding to $\lambda = (\lambda_1\cdots, \lambda_n)$.
Then its length is
\begin{equation}\label{eq-norm 1-su(n)}
	\tau_S(\pi_\lambda) = \norm{\pi_\lambda}_1 = \lambda_1,
	\end{equation}
and its character function $\chi_\lambda = \chi_{\pi_\lambda}$ has the following form when it is restricted to a maximal torus $H_n = \{\text{diag}(x_1,\cdots, x_n)\}$:
	$$\chi_\lambda(x_1,\ldots,x_n)=\sum_{T}x_1^{t_1}\ldots x_n^{t_n},$$
where $T$ runs through all the semistandard Young tableaux of shape $\lambda$ with parameters $t_1,\ldots, t_n$. Here the parameter $t_k$, $1\le k \le n$, is the number of times $k$ appear in the tableau. Moreover,
	$$\sum_{i=1}^n t_i=\sum_{i=1}^{n-1} \lambda_i$$
and we have the following dimension formula:
	$$d_\lambda = d_{\pi_\lambda} = \prod_{1\leq i<j\leq n }\frac{\lambda_i-\lambda_j + j-i}{j-i}.$$
Since $x_1\cdots x_n = 1$, we may also write as follows
\begin{equation}\label{eq-schur poly}
\chi_\lambda(x_1,\ldots,x_{n-1})=\sum_{T}x_1^{t_1-t_n}\ldots x_{n-1}^{t_{n-1}-t_n}.
\end{equation}
Now we turn our attention to the restriction of weights to subgroups.
Let $H$ be a closed subgroup of $G$, and let $\om : \widehat{G} \to [1,\infty)$ be a weight.
We define the restriction of $\om$ on $H$, $\om_H : \widehat{H} \to [1,\infty)$ by
\begin{equation}\label{eq-restriction-defn}
	\om_H(\pi)=\inf \{\om(\widetilde{\pi}) \mid \pi \subset \widetilde{\pi}|_H\}.
	\end{equation}
Then it is shown in \cite[Proposition 3.5]{LS} that $\om_H$ is a weight on $\widehat{H}$ and the restriction map
	\begin{equation}\label{eq-restriction}
	R_H: A(G,\om) \to A(H,\om_H), \;\; f \mapsto f|_{H}
	\end{equation}
is a complete quotient map.

We mainly focus on the case when $G=SU(n)$ and $H = H_n\cong \mathbb{T}^{n-1}$, a maximal torus of $SU(n)$.
We first need to understand the decomposition of $\pi|_{H_n}$ for any $\pi\in \widehat{G}$.
The following proposition is an immediate consequences of the definition of $\pi_\lambda$.

	\begin{prop}\label{prop-subrep}
	Let $\lambda$ and $\pi_\lambda$ be as above and $P = (p_1,\cdots,p_{n-1})\in \z^{n-1}$ be an $(n-1)$-tuple of integers.
	Then, the character $\chi_P$ of $\mathbb{T}^{n-1}$ associated to $P$ satisfies
		$$\chi_P \subset \pi_\lambda|_{H_n}$$
	if and only if there exists a semistandard Young tableau $T$ with parameters $t_1,\ldots, t_n$ such that
	$t_i-t_n=p_i$ for every $1\leq i\leq n-1$.
	\end{prop}

We show in the following theorem that if we restrict dimension weights on $\widehat{SU(n)}$ down to $\widehat{H_n}\cong \mathbb{Z}^{n-1}$, then we would again get polynomial weights on $\mathbb{Z}^{n-1}$. Recall the polynomial weight $\rho_{\alpha}$ of order $\alpha>0$ on $\mathbb{Z}^{n-1}$ is given by
	$$\rho_\alpha(P) = (1+\norm{P})^\alpha =  (1+\abs{p_1}+\cdots+\abs{p_{n-1}})^\alpha,\;\; P=(p_1,\ldots, p_{n-1})\in \mathbb{Z}^{n-1}.$$

	\begin{thm}\label{thm-restriction-dim}
      Let $\alpha \geq 0$, and let $\om_\alpha$ the dimension weight on $\widehat{SU(n)}$ defined in Definition \ref{def-dimension-weight}. Then the restriction of $\om_\alpha$ to $\widehat{H_n}$ is equivalent to
	the polynomial weight $\rho_{(n-1)\alpha}$ on $\mathbb{Z}^{n-1}$ up to constants depending only on $n$ and $\alpha$.
	Moreover, $A(\mathbb{T}^{n-1}, \rho_{(n-1)\alpha})$ is completely isomorphic with the complete quotient of
	$A({SU(n)},\om_\alpha)$ coming from the restriction to $H_n$.
	\end{thm}
	
\begin{proof}
Fix $n\geq 2$ and $P=(p_1,\ldots, p_{n-1})\in {\mathbb Z}^{n-1}$. By \eqref{eq-restriction-defn} and
Proposition \ref{prop-subrep}, we must estimate the infimum of $d_\lambda$ over all possible $\chi_P \subset \pi_\lambda|_{H_n} \in \widehat{SU(n)}$.
Without loss of generality, we can assume that $p_1\geq p_2\geq\ldots\geq p_{n-1}$ since
the Schur polynomial is symmetric (so the rearrangement $t'_1\geq t'_2\geq \ldots\geq t'_{n-1}, t_n$ of $t_1,\ldots, t_{n-1}, t_n$
also appears in a semistandard Young tableau of shape $\lambda$),
the latter is equivalent to $\chi_{P'}\subset \pi_\lambda|_{{\mathbb T}^{n-1}}$,
where $P'$ is the rearrangement of $P$ in the non-increasing order.

We now consider a particular $\lambda = \lambda_P\in \z^n_+$ such that $\chi_P\subset \pi_\lambda|_{{\mathbb T}^{n-1}}$ given by
	\begin{equation}\label{eq-choice-young}
	\lambda_{P} : \lambda_1=\sum_{i=1}^{n-1}p_i+n|p_{n-1}|,\; \lambda_2=\cdots = \lambda_n=0.
	\end{equation}
If we set the parameters $t_1,\cdots, t_n$ by	
	\begin{eqnarray}\label{eq-choice-t}
	t_n&=&|p_{n-1}|,\\
	t_i&=&p_i+|p_{n-1}|, \ 1\leq i\leq n-1,\nonumber
	\end{eqnarray}
then since $p_1\geq p_2\geq \ldots\geq p_{n-1}$, we have
	$$t_1\geq t_2\geq \ldots\geq t_{n-1}=p_{n-1}+|p_{n-1}|\ge 0.$$
Therefore for each $1\leq i\leq n$, we have $t_i\geq 0$.
Note that $\lambda_P$ is a diagram with only one row,
so that it is easy to find a semistandard Young tableau of shape $\lambda_P$ in which
the weight of each integer $j$ is exactly $t_j$ for every $1\leq j\leq n$ as follows:
	$$\underbrace{1\dots\dots1}_{t_1}\underbrace{2\dots\dots2}_{t_2}\ldots\underbrace{n\dots\dots n}_{t_n}.$$
Moreover,
	\begin{align*}
	d_{\lambda_P}& = \frac{\prod_{1\leq i<j\leq n}(\lambda_i-\lambda_j + j-i)}{\prod_{1\le i<j\leq n}(j-i)}\\
	&= \frac{1}{(n-1)!}\prod_{1<j\le n}\left(\sum_{i=1}^{n-1}p_i+n|p_{n-1}| + j-1\right)\\
	&\le \frac{1}{(n-1)!}\prod_{1<j\le n}\left[(n+1)\left(\sum_{i=1}^{n-1} \abs{p_i}+1\right)\right]\\
	&= \frac{(n+1)^{n-1}}{(n-1)!}\left(\sum_{i=1}^{n-1} \abs{p_i} + 1\right)^{n-1}.
	\end{align*}
On the other hand, let $\lambda$ be any dominant weight such that $\chi_{P}\subset \pi_\lambda|_{{\mathbb T}^{n-1}}$.
Then there exist parameters $t_1,\ldots,t_n$ such that $p_i=t_i-t_n$ for every $1\leq i\leq n-1$. Moreover, we have
	\begin{align}\label{eq-lower-estimate}
	\sum_{i=1}^{n-1}|p_i| + 1 & = \sum_{i=1}^{n-1}|t_i-t_n| + 1 \le \sum_{i=1}^{n-1}(t_i+t_n) + 1\\
	& \le 1+ (n-1)\sum_{i=1}^nt_i = 1+(n-1)\sum_{i=1}^{n-1} \lambda_i \nonumber \\
	& \le 1+n(n-1) \lambda_1 \leq n^2(\lambda_1+1).\nonumber
	\end{align}
We also have
\begin{align*}
\lefteqn{\prod_{1\leq i<j\leq n}(\lambda_i-\lambda_j + j-i)}\\
&\ge (\lambda_1+1)(\lambda_1-\lambda_2 + 1)(\lambda_2+1)\cdots(\lambda_1-\lambda_{n-1} +1)(\lambda_{n-1} + 1)\\
&\ge (\lambda_1+1)\left(\frac{\lambda_1 + 2}{2}\right)^{n-2} \ge \frac{1}{2^{n-2}}(\lambda_1+1)^{n-1},
\end{align*}
where we used the fact that $ab\geq\frac{a+b}{2}$ whenever $a$ and $b$ are both at least 1.
By combining the preceding inequality with \eqref{eq-lower-estimate}, we get
	\begin{align*}
	d_\lambda & = \frac{\prod_{1\leq i<j\leq n}(\lambda_i-\lambda_j + j-i)}{\prod_{1\le i<j\leq n}(j-i)}
	\ge \frac{(\lambda_1+1)^{n-1}}{2^{n-2}\prod_{1\le i<j\leq n}(j-i)}\\
	& \ge c_n\left(\sum_{i=1}^{n-1}|p_i|+1\right)^{n-1},
	\end{align*}
where $\displaystyle c_n = \frac{1}{(n^2)^{n-1}2^{n-2}\prod_{1\le i<j\leq n}(j-i)}$.
Consequently, we have
	$$c_n\left(\sum_{i=1}^{n-1}|p_i|+1\right)^{n-1}\ \le \om_1|_{\widehat{H}_n}(\chi_P)
	\le d_n\left(\sum_{i=1}^{n-1}|p_i|+1\right)^{n-1},$$
where $\displaystyle d_n = \frac{(n+1)^{n-1}}{(n-1)!}$. Since $n\geq 2$ and $P=(p_1,\ldots, p_{n-1})\in {\mathbb Z}^{n-1}$
are arbitrary,  we conclude that for every $\alpha \geq 0$,
\begin{align*}\label{eq-dim weight-res-estim}
c^\alpha_n\rho_{(n-1)\alpha} \le \om_\alpha|_{\widehat{H}_n}
	\le d^\alpha_n \rho_{(n-1)\alpha}.
	\end{align*}
	The final result follows from \eqref{eq-restriction} (see also \cite[Proposition 3.5]{LS}).

\end{proof}	
	
We can also make similar estimation for the restriction of weights of polynomial type
on $\widehat{SU(n)}$ down to $\widehat{H_n}\cong \mathbb{Z}^{n-1}$. We will again obtain polynomial weights on $\mathbb{Z}^{n-1}$. However, the order will be different and computation become more straightforward.
	
	\begin{thm}\label{thm-restriction-poly}
	Let $\alpha \geq 0$, and let $\om^\alpha_S$ the polynomial weight on $\widehat{SU(n)}$ defined in Definition \ref{def-poly-weight}. Then the restriction of $\om^\alpha_S$ to $\widehat{H_n}$ is equivalent to
	the polynomial weight $\rho_{\alpha}$ on $\mathbb{Z}^{n-1}$ up to constants depending only on $n$ and $\alpha$.
	Moreover, $A(\mathbb{T}^{n-1}, \rho_{\alpha})$ is completely isomorphic with the complete quotient of
	$A({SU(n)},\om^\alpha_S)$ coming from the restriction to $H_n$.
	\end{thm}
	
\begin{proof}
By Remark \ref{re-lenght-norm 1-equ} and \eqref{eq-norm 1-su(n)}, we can assume that for every
$\pi_\lambda \in \widehat{SU(n)}$,
$$ \om^\alpha_S(\pi_\lambda)=(1+\lambda_1)^\alpha.$$
Fix again $n\geq 2$ and $P=(p_1,\ldots, p_{n-1})\in {\mathbb Z}^{n-1}$. As in the proof of Theorem \ref{thm-restriction-dim}, we can assume that
	$$p_1\geq p_2\geq\ldots\geq p_{n-1},$$
and consider a particular $\lambda = \lambda_P\in \z^n_+$ such that $\chi_P\subset \pi_\lambda|_{{\mathbb T}^{n-1}}$ by assigning the same parameters $\lambda_i$ and $t_i$ as in \eqref{eq-choice-young} and \eqref{eq-choice-t}. Then
	$$1+\lambda_1 = 1+\sum_{i=1}^{n-1}p_i+n|p_{n-1}| \le (n+1) \left(\sum_{i=1}^{n-1}|p_i|+1\right).$$
On the other hand, let $\lambda$ be any dominant weight such that $\chi_{P}\subset \pi_\lambda|_{{\mathbb T}^{n-1}}$.
Then we can use \eqref{eq-lower-estimate} to get
	$$1+\lambda_1 \ge \frac{1}{n^2} \left(\sum_{i=1}^{n-1}|p_i|+1\right).$$
	Putting together the preceding two inequalities, \eqref{eq-restriction-defn} and
Proposition \ref{prop-subrep}, we have
	$$ \frac{1}{n^2} \left(\sum_{i=1}^{n-1}|p_i|+1\right) \le \om^1_S|_{\widehat{H}_n}(\chi_P) \le (n+1) \left(\sum_{i=1}^{n-1}|p_i|+1\right).$$
	Since $n\geq 2$ and $P=(p_1,\ldots, p_{n-1})\in {\mathbb Z}^{n-1}$
are arbitrary,  we conclude that for every $\alpha \geq 0$,
	$$\frac{1}{n^{2\alpha}} \rho_{\alpha} \le \om^\alpha_S|_{\widehat{H}_n} \le (n+1)^\alpha \rho_{\alpha}.$$
The final result again follows from \eqref{eq-restriction} (see also \cite[Proposition 3.5]{LS}).	
\end{proof}

In the above we get equivalence of weights since we are working on polynomial types of weights. When we deal with exponential type of weights the above restriction does not guarrantee the equivalence of weights. However, restricting further down to 1-dimensional torus allows us to get an exact formula, which will help us later in section \ref{S:Exp. weight-poly growth}.

	\begin{thm}\label{thm-restriction-exp}
	 Let $\gamma^1_S$ the exponential weight on $\widehat{SU(n)}$ defined in Definition \ref{def-poly-weight}. Let $T$ be the 1-dimensional torus in $H_n$ whose entries are all 1 except for the first two. Then the restriction of $\gamma^1_S$ to $\widehat{T}$ is exactly the same as the exponential weight $e^{|\cdot|}$ on $\mathbb{Z}$. Moreover, $\ell^1(\mathbb{Z}, e^{|\cdot|})$ is a complete quotient of $A({SU(n)},\gamma^1_S)$.
	\end{thm}
	
\begin{proof}
The same approach as in Theorem \ref{thm-restriction-poly} gives us the conclusion. Indeed, we begin with $P=(p_1, 0, \dots, 0)\in {\mathbb Z}^{n-1}$. Now we set
	$$t_1 = \lambda_1 = \abs{p_1},\; \lambda_2 = t_2 = \dots = \lambda_n = t_n = 0,$$
then it is easy to observe that this choice of parameters can be easily realized in a semistandard Young tableau of shape $\lambda_P$. Thus, we have
	$$\gamma^1_S|_{T}(P) \le e^{\abs{p_1}}.$$
For the converse direction we let $\lambda$ be any dominant weight such that $\chi_{P}\subset \pi_\lambda|_{{\mathbb T}^{n-1}}$.
Then there exist parameters $t_1, \dots, t_n$ such that
	$$p_1 = t_1 - t_n\;\;\text{and}\;\; 0 = t_2 - t_n = \dots = t_{n-1} - t_n.$$
Since the parameters should be realized  a semistandard Young tableau of shape $\lambda$ we clearly have that $t_1, t_n \le \lambda_1$, which implies that $\abs{p_1} \le \lambda_1$. Thus, we have
	$$\gamma^1_S|_{T}(P) \ge e^{\abs{p_1}}.$$
\end{proof}

\section{Beurling--Fourier algebras on compact groups which are operator algebras}\label{sec-BF-OP}

In this section, we investigate when a Beurling--Fourier algebra on a compact connected Lie group
can be completely boundedly isomorphic to an operator algebra. Throughout this section, we use the
term ``positive result" when such a thing happens and ``negative result" when it does not.

Our approach for seeking Beurling--Fourier algebras as operator algebras is based on the following theorem of Blecher (\cite{B95}).

	\begin{thm}
	Let $\A$ be a completely contractive Banach algebra with the algebra multiplication
		$$m : \A \prt \A \to \A,$$
where $\prt$ is the projective tensor product of operator spaces. Then, $\A$ is completely isomorphic to an operator algebra if and only if the multiplication map extends to a completely bounded map
		$$m : \A \otimes_h \A \to \A.$$
	\end{thm}

In the case of $\A = A(G,\om)$ with the operator space structure described in section \ref{sec-BF-alg-gen},
$A(G,\om)$ is completely isomorphic to an operator algebra if and only if the following map is completely bounded.
	\begin{equation}\label{eq-Op-alg}
	\widetilde{\Gamma} : VN(G) \rightarrow VN(G)\otimes_{eh} VN(G),\;
	A \mapsto \Gamma(A)\Gamma(W)(W^{-1}\otimes W^{-1}).
	\end{equation}
Since we already know $\Gamma : VN(G) \rightarrow VN(G)\vnt VN(G)$ is a complete contraction, we can get the positive direction
(i.e. $A(G,\om)$ being completely isomorphic to an operator algebra)
if $\Gamma(W)(W^{-1}\otimes W^{-1})$ can be split as a sum of right or left cb-multiplier
from $VN(G)\vnt VN(G)$ into $VN(G)\otimes_{eh} VN(G)$,
where we could apply non-commutative Littlewood multiplier theory we developed earlier.

The following lemma will be used frequently throughout this section.

\begin{lem}\label{lem-summability}
	We have $\displaystyle \sum_{i\in \mathbb{Z}^n} \frac{1}{(1+\norm{i})^{2\alpha}}<\infty$ if and only if
	$\displaystyle \sum_{i\in \mathbb{Z}^n_+} \frac{1}{(1+\norm{i})^{2\alpha}}<\infty$ if and only if $\alpha > \frac{n}{2}$.
	\end{lem}
\begin{proof}
The above series is sometimes called an Epstein series (\cite[p.277]{Far} for example). The results follows from a standard integral test argument.
\end{proof}

\subsection{The case of $\mathbb{T}^n$ with polynomial weights}
In this subsection, we consider the case of $G = \mathbb{T}^n$ with polynomial weights.
Since
	$$\widehat{\mathbb{T}^n} = \mathbb{Z}^n \;\;\text{and}\;\; A(\mathbb{T}^n, \om) \cong \ell^1(\mathbb{Z}^n, \om)$$
we can reformulate our problem as follows.

{\it The weighted convolution algebra $\ell^1(\mathbb{Z}^n, \om)$ with the maximal operator space structure is
completely isomorphic to an operator algebra if and only if the following map is completely bounded.
	\begin{equation}\label{eq-Op-alg-abelian}
	\widetilde{\Gamma} : \ell^\infty(\mathbb{Z}^n) \to \ell^\infty(\mathbb{Z}^n) \otimes_{eh} \ell^\infty(\mathbb{Z}^n),\;
	(a_k)_{k \in \mathbb{Z}^n} \mapsto (T(i,j) a_{i+j})_{i,j \in \mathbb{Z}^n},
	\end{equation}
where $T = (T(i,j))_{i,j \in \mathbb{Z}^n}$ is the matrix given by
\begin{equation}\label{eq-co prod-weight}
T(i,j) = \frac{\om(i+j)}{\om(i) \om(j)}
\end{equation}
associated with the weight $\om : \mathbb{Z}^n \to [1,\infty)$.
}

We will present a complete solution focusing on the case of polynomial weight $\rho_\alpha$.
Note that the 1-dimensional case has already been established in \cite{Var72} in the setting of Banach spaces.
The authors thank \'{E}ric Ricard for providing the main idea of the proof.

For the proof we need some background material of harmonic analysis. Let
	$$Q : L^\infty(\mathbb{T}) \to B(\ell^2), \;\; f \mapsto (\widehat{f}(-(i+j)))_{i,j\in \z}.$$
According to Nehari's theorem $Q$ is a contractive surjection onto the space of Hankelian matrices (see \cite[Section 6]{Pis} for example).

One more ingredient is the Rudin-Shapiro polynomials.
Recall that the Rudin-Shapiro polynomials are defined in the following recursive way.
	$$P_0(z) := 1,\;\; Q_0(z) :=1,$$
and for $k\ge 0$,
	$$P_{k+1}(z) := P_k(z) + z^{2^k}Q_k(z),\;\; Q_{k+1}(z) := Q_k(z) - z^{2^k}P_k(z).$$
By doing an induction on $k$, it is straightforward to check that the coefficients of $P_k$ are $\pm 1$, $\deg P_k =\deg Q_k=2^k-1$ and
$$ |P_k(z)|^2+|Q_k(z)|^2=2^{k+1} \ \ \ (z\in \mathbb{T}).$$ Hence
	$$\norm{P_k}_{L^\infty(\mathbb{T})} \le \sqrt{2^{k+1}}.$$
Combining the above two ingredients we get a sequence of Hankelian matrices
	$$A_{2^k} = Q(\overline{P}_k),\; k\ge 0,$$ where $A_{2^k}$ is a $2^k\times 2^k$ matrix with entries $\pm 1$ satisfying
	\begin{equation}\label{eq-Shapiro-matrix}
		\norm{A_{2^k}}_\infty \le \sqrt{2^{k+1}}.
	\end{equation}

	\begin{thm}\label{thm-torus}
	The weighted convolution algebra $\ell^1(\mathbb{Z}^n, \rho_\alpha)$, $\alpha>0$ with the maximal operator space structure is
	completely isomorphic to an operator algebra if and only if $\alpha > \frac{d(\mathbb{T}^n)}{2} = \frac{n}{2}$.
	\end{thm}
\begin{proof}
Let $T^\alpha$ be the matrix \eqref{eq-co prod-weight} associated to $\rho_\alpha$, which means
	$$T^\alpha(i,j) = \left( \frac{1+\norm{i+j}}{(1+\norm{i})(1+\norm{j})}\right)^\alpha.$$
We need to determine for which value of $\alpha$, the mapping $\widetilde{\Gamma}$ defined in \eqref{eq-Op-alg-abelian} is completely bounded.
Clearly we have
	$$T^\alpha(i,j) \le \left( \frac{1}{1+\norm{i}} + \frac{1}{1+\norm{j}} \right)^\alpha
	\le 2^\alpha\left( \frac{1}{(1+\norm{i})^\alpha} + \frac{1}{(1+\norm{j})^\alpha} \right),$$
so that $T^\alpha = S\widetilde{T_1^\alpha}+S\widetilde{T_2^\alpha}$, where
 $$\widetilde{T_1^\alpha}(i,j)=\frac{1}{(1+\norm{i})^{2\alpha}} \ \ \ , \ \ \
 \widetilde{T_2^\alpha}(i,j)=\frac{1}{(1+\norm{j})^{2\alpha}}$$
and $S\in \ell^\infty(\mathbb{Z}^n \times \mathbb{Z}^n)$ with $0 <S \leq 2^\alpha$. Thus $T^\alpha \in \mathcal{T}^2=\mathcal{T}^2_r + \mathcal{T}^2_c$ provided that
	$$\sum_{i\in \mathbb{Z}^n} \frac{1}{(1+\norm{i})^{2\alpha}}<\infty.$$
Now we have the positive result for $\alpha > \frac{n}{2}$ by Theorem \ref{thm-multiplier} and Lemma \ref{lem-summability}.

For the negative direction, we consider the restricted sequence $T^\alpha_d$ of $T^\alpha$ to the set of indices $I^n_d \times I^n_d$, where
$$I^n_d = \{i = (i_1,\cdots,i_n) : 1\le i_1,\cdots, i_n \le d\}.$$
Then we may regard $T^\alpha_d$ as a matrix acting on $\ell^2(I^n_d)$, and the norm of $T^\alpha_d$ in
$\ell^\infty(\mathbb{Z}^n) \otimes_{eh} \ell^\infty(\mathbb{Z}^n)$ is exactly the Schur norm of $T^\alpha_d$ \cite[Theorem 3.1]{Spr1}.
Moreover, we have a lower estimate of the operator norm of $T^\alpha_d$ as follows.
	\begin{equation}\label{eq-op-norm}
	\norm{T^\alpha_d}_\infty \ge 2^{-\alpha}d^{\frac{n}{2}}
	\left(\sum_{i\in I^n_d} \frac{1}{(1+\norm{i})^{2\alpha}}\right)^{\frac{1}{2}}.
	\end{equation}
Indeed, if we set $\displaystyle v=\sum_{i\in I^n_d} e_i \in \ell^2(I^n_d)$, then $\|v\|_2=d^{\frac{n}{2}}$ and
	\begin{align*}
	\norm{T^\alpha_d v}_2
	& = \norm{\left[\sum_{j\in I^n_d}\left( \frac{1+\norm{i+j}}{(1+\norm{i})(1+\norm{j})}\right)^\alpha \right]_{i\in I^n_d}}_2\\
	& \ge \norm{\left[2^{-\alpha}
	\sum_{j\in I^n_d}\left(\frac{1}{1+\norm{i}}+\frac{1}{1+\norm{j}}\right)^\alpha \right]_{i\in I^n_d}}_2\\
	& \ge 2^{-\alpha}\norm{\left[\sum_{j\in I^n_d}\left(\frac{1}{1+\norm{i}}\right)^\alpha \right]_{i\in I^n_d}}_2\\
	& =  2^{-\alpha}d^n\left(\sum_{i\in I^n_d} \frac{1}{(1+\norm{i})^{2\alpha}}\right)^{\frac{1}{2}}.
	\end{align*}

Now we recall a sequence of Hankelian matrices $A_d\in M_d$, $d = 2^k\ge 1$ in \eqref{eq-Shapiro-matrix}. Then we have
	$$A_d = (a_{i+j})^d_{i,j=1}$$
with $a_i \in \{\pm 1\}$.
Let $\displaystyle a = \sum_{i=1}^d a_i \delta_i \in \ell^\infty(I^1_d)$ and
	$$b=a\otimes \cdots \otimes a \in \ell^\infty(I^n_d) \subset \ell^\infty(\z^n),$$
i.e. $b_{i_1,\cdots,i_n} = a_{i_1}\cdots a_{i_n}$. Then the associated Hankel matrix
of b, i.e.	$$B= (b_{i+j})_{i,j\in I^n_d},$$
is nothing but $B = A_d\otimes \cdots \otimes A_d$, the $n$-tensor power of $A_d$.
Since we have
	$$\widetilde{\Gamma}(b) = \left[b_{i+j}T^\alpha_d(i,j)\right]_{i,j\in I^n_d}$$
and each $b_i = \pm 1$ we get
	$$\norm{T^\alpha_d}_\infty \le \norm{B}_\infty \cdot \norm{\left[b_{i+j}T^\alpha_d(i,j)\right]_{i,j\in I^n_d}}_{\ell^\infty(I^n_d) \otimes_{eh} \ell^\infty(I^n_d)}.$$
Then, by \eqref{eq-Shapiro-matrix}, $\|B\|_\infty \leq \|A\|_\infty^n \leq (2d)^{\frac{n}{2}}$,
we get
	\begin{align*}
	\norm{\widetilde{\Gamma}} & \ge
\norm{\left[b_{i+j}T^\alpha_d(i,j)\right]_{i,j\in I^n_d}}_{\ell^\infty(I^n_d) \otimes_{eh} \ell^\infty(I^n_d)} \\
	& \ge \norm{B}^{-1}_\infty \norm{T^\alpha_d}_\infty\\
	& \ge d^{-\frac{n}{2}}2^{-\alpha-\frac{n}{2}}d^{\frac{n}{2}}
	\left(\sum_{i\in I^n_d} \frac{1}{(1+\norm{i})^{2\alpha}}\right)^{\frac{1}{2}} \ \ (\text{by}\ \eqref{eq-op-norm})\\
	& = 2^{-\alpha-\frac{n}{2}}\left(\sum_{i\in I^n_d} \frac{1}{(1+\norm{i})^{2\alpha}}\right)^{\frac{1}{2}}.
	\end{align*}
If $\alpha \le \frac{n}{2}$, then the right-hand side grows without bounds when $d\to \infty$ (Lemma \ref{lem-summability}),
so that we have the negative result.	
\end{proof}

In the proof of the above theorem we actually showed that the pointwise multiplication $A(\mathbb{T}^n, \rho_\alpha) \otimes_h A(\mathbb{T}^n, \rho_\alpha) \to A(\mathbb{T}^n, \rho_\alpha)$ is not even {\it bounded} for $\alpha \le \frac{n}{2}$, which allows us to extend the result of Varapolous to the multi-dimensional situation as follows.
\begin{cor}\label{cor-Varopoulos-extended}
The Beurling--Fourier algebra $A(\mathbb{T}^n, \rho_\alpha)\cong \ell^1(\z^n, \rho_\alpha)$, $n\ge 1$ is a $Q$-algebra if and only if $\alpha > \frac{n}{2}$.
\end{cor}
\begin{proof}
We first recall some related notions. We say that a Banach algebra $\A$ is called an {\it injective algebra} if the algebra multiplication $m:\A \otimes_\gamma \A \to \A$ extends to a bounded map $m: \A \otimes_\eps \A \to \A$, where $\otimes_\gamma$ and $\otimes_\eps$ refer to projective and injective tensor products of Banach spaces, respectively. In \cite{Var72b} Varapolous introduced the notion of $DQ$-algebra (direct $Q$-algebra). We say that a commutative Banach algebra $\A$ is a $DQ$-algebra if there is a uniform algebra $\B$, a Banach algebra homomorphism $\varphi : \B \to \A$ and a bounded linear map $\psi : \A \to \B$ such that $\varphi\circ \psi = id_\A$. The main theorem in the same paper says that for a commutative Banach algebra $\A$ we have $\A$ is injective if and only if $\A$ is a $DQ$-algebra. The Beurling--Fourier algebra $A(\mathbb{T}^n, \rho_\alpha)\cong \ell^1(\z^n, \rho_\alpha)$ is a $\ell^1$-space, so that we have two advantages. First, $A(\mathbb{T}^n, \rho_\alpha)$ is an injective algebra if and only if the pointwise multiplication $A(\mathbb{T}^n, \rho_\alpha) \otimes_h A(\mathbb{T}^n, \rho_\alpha) \to A(\mathbb{T}^n, \rho_\alpha)$ extends to a {\it bounded} map by Grothendieck's theorem (see \cite[(1.47), (A.7)]{BLe04} and \cite[(3.11)]{Pis3}). Secondly, $A(\mathbb{T}^n, \rho_\alpha)$ is a $DQ$-algebra if and only if it is a $Q$-algebra, as follows from the lifting property of $\ell^1$-spaces \cite{Gro}. Recall that a Banach space $X$ is said to have the lifting property if for any bounded linear map $u:X \to Y/Z$ and $\eps>0$ there is a bounded linear map $v:X\to Y$ such that $q\circ v=u$ and $\norm{v}\le (1+\eps)\norm{u}$, where $q:Y\to Y/Z$ is the canonical quotient map. Now the proof of Theorem \ref{thm-torus} tells us that  $A(\mathbb{T}^n, \rho_\alpha)$ is an injective algebra if and only if $\alpha > \frac{n}{2}$. Indeed, we have shown in the proof of Theorem \ref{thm-torus} that the pointwise multiplication $A(\mathbb{T}^n, \rho_\alpha) \otimes_h A(\mathbb{T}^n, \rho_\alpha) \to A(\mathbb{T}^n, \rho_\alpha)$ is completely bounded for $\alpha > \frac{n}{2}$ and not bounded for $\alpha \le \frac{n}{2}$. Combining the above observations we get the desired conclusion.
\end{proof}

\subsection{The general case of compact connected Lie groups with polynomial weights}

When the group is a compact connected Lie group, then we have positive results if the order of the polynomial weight is strictly greater than the half of the dimension of the group.

	\begin{thm}\label{thm-general-poly}
	Let $G$ be a compact connected Lie group, and let $\om^\alpha_S$ be the polynomial weight of order $\alpha$ on $\widehat{G}$ $($Definition \eqref{def-poly-weight}$)$. Then $A(G, \om^\alpha_S)$ is completely isomorphic to an operator algebra
	if $\alpha>\frac{d(G)}{2}$.
	\end{thm}
\begin{proof}

Suppose that $\alpha>\frac{d(G)}{2}$. For simplicity, we write $\om$ instead of $\om^\alpha_S$. We set $W = \bigoplus_{\pi\in \widehat{G}} \om(\pi){\rm id}_{M_{d_\pi}}\in VN(G)$ be the operator associated to $\om$  and
	$$T = \Gamma(W)(W^{-1}\otimes W^{-1}) \in VN(G)\bar{\otimes}VN(G).$$
Then, by \eqref{eq-weight3}, we have
	\begin{equation*}\label{eq-T}
	T(\pi, \pi') = \bigoplus_{\sigma \subset \pi\otimes \pi'} \left(\frac{1+\tau_S(\sigma)}{(1+\tau_S(\pi))(1+\tau_S(\pi'))}\right)^\alpha {\rm id}_{M_{d_\sigma}}.
	\end{equation*}
By \eqref{eq-Op-alg}, we need to show that $\widetilde{\Gamma} : A \mapsto \Gamma(A)T$ is well-defined and completely bounded. To achieve this, we will apply the non-commutative Littlewood machinery developed in Section \ref{sec-Littlewood} with $L^\infty = VN(G)$ to get the decomposition of the operator $T$ into $T = T_1 + T_2$ with $T_1 \in \mathcal{T}^2_r$ and $T_2\in \mathcal{T}^2_c$. In order to do so, we first need to estimate each component of $T$ as follows.

Let $\pi, \pi' \in \widehat{G}$ and $\sigma \subset \pi\otimes \pi'$. Then, by
\eqref{eq2}, we have
	\begin{align*}
	 & \left(\frac{1+\tau_S(\sigma)}{(1+\tau_S(\pi))(1+\tau_S(\pi'))}\right)^\alpha\\
	& \le (1+C)^{-2\alpha} \left(\frac{1+\norm{\sigma}_1}{(1+\norm{\pi}_1)(1+\norm{\pi'}_1)}\right)^\alpha\\
	& \le (1+C)^{-2\alpha} \left(\frac{1+\norm{\pi}_1+\norm{\pi'}_1}{(1+\norm{\pi}_1)(1+\norm{\pi'}_1)}\right)^\alpha\\
	& \le \left(\frac{2}{(1+C)^2}\right)^\alpha\left(\frac{1}{(1+\norm{\pi}_1)^\alpha}+\frac{1}{(1+\norm{\pi'}_1)^\alpha}\right).	 \end{align*}
Hence
	$$T =S (\tilde{T_1}+\tilde{T_2}),$$
where $\tilde{T_1}, \tilde{T_2}\in VN(G) \overline{\otimes} VN(G)$ are positive and central elements defined by
	$$\tilde{T}_1(\pi, \pi')  = \frac{1}{(1+\norm{\pi}_1)^\alpha} {\rm id}_{M_{d_\pi}} \otimes {\rm id}_{M_{d_{\pi'}}}$$
and
	$$\tilde{T}_2(\pi, \pi')  = \frac{1}{(1+\norm{\pi'}_1)^\alpha} {\rm id}_{M_{d_\pi}} \otimes {\rm id}_{M_{d_{\pi'}}},$$
and $S\in VN(G) \overline{\otimes} VN(G)$ is some positive element with $\|S\|\leq \left(\frac{2}{(1+C)^2}\right)^\alpha$.
We claim that $\tilde{T}_1 \in \mathcal{T}^2_r$, $\tilde{T}_2\in \mathcal{T}^2_c$.
Indeed
	$$\tilde{T}_1 = \left(\bigoplus_{\pi\in \widehat{G}}
	\frac{1}{(1+\norm{\pi}_1)^\alpha}{\rm id}_{M_{d_\pi}}\right) \otimes 1_{VN(G)}$$
and
	$$\tilde{T}_2 = 1_{VN(G)} \otimes
	\left(\bigoplus_{\pi'\in \widehat{G}}\frac{1}{(1+\norm{\pi'}_1)^\alpha} {\rm id}_{M_{d_{\pi'}}}\right).$$
Hence
	$$\norm{\tilde{T}_1}_{\mathcal{T}^2_r}
	=  \norm{\bigoplus_{\pi\in \widehat{G}}\frac{1}{(1+\norm{\pi}_1)^\alpha} {\rm id}_{M_{d_\pi}}}_{L^2}
	=  \left(\sum_{\pi\in \widehat{G}}\frac{d^2_\pi}{(1+\norm{\pi}_1)^{2\alpha}}\right)^{\frac{1}{\alpha}},$$
so that $\tilde{T}_1 \in \mathcal{T}^2_r$ since $\alpha>\frac{d(G)}{2}$ (see \cite[Lemma 5.6.7]{Wal73}).
Similarly, $\tilde{T}_2 \in \mathcal{T}^2_c$.
Now, by the centrality of $\tilde{T_2}$, we have that for any $A \in VN(G)$,
	$$\tilde{\Gamma}(A) = \Gamma(A)T = \Gamma(A)S\tilde{T}_1 + \Gamma(A)S\tilde{T}_2
	= \Gamma(A)S\tilde{T}_1 + \tilde{T}_2\Gamma(A)S.$$
Since the maps
	$$VN(G) \overline{\otimes}VN(G) \to VN(G) \otimes_{eh} VN(G),\;\; X \mapsto X S\tilde{T}_1$$
and	$$VN(G) \overline{\otimes}VN(G) \to VN(G) \otimes_{eh} VN(G),\;\; X \mapsto \tilde{T}_2XS$$
are completely bounded by Corollary \ref{cor-moreopmod} and Theorem \ref{thm-multiplier},
we can conclude that $\tilde{\Gamma}$ is also completely bounded.

\end{proof}

In general, we were not able to obtain the negative result for $A(G, \om^\alpha_S)$. In fact, we believe this to be very difficult. However, in the special case when $G = SU(n)$, we have the following:

	\begin{thm}\label{thm-SU(n)-poly-negative}
	The Beurling--Fourier algebra $A(SU(n), \om^\alpha_S)$ is not completely isomorphic to an operator algebra if $\alpha\le\frac{n-1}{2}$.
	\end{thm}

\begin{proof}
It follows from Theorem \ref{thm-torus} that $A(\mathbb{T}^{n-1}, \rho_{\alpha})$ is not completely isomorphic to an operator algebra if $\alpha\le\frac{n-1}{2}$. On the other hand, by Theorem \ref{thm-restriction-poly}, $A(\mathbb{T}^{n-1}, \rho_{\alpha})$ is completely isomorphic to a
complete quotient of $A(SU(n), \om^\alpha_S)$. Hence the result follows from the fact that
a complete quotient of an operator algebra is again an operator algebra \cite[Proposition 2.3.4]{BLe04}.
\end{proof}

	\begin{rem}\label{R:BF SU(n)-not Q alg}{\rm
		\begin{enumerate}
			\item
			Theorem \ref{thm-torus} tells us that the exponent $\frac{d(G)}{2}$ is optimal when $G = \mathbb{T}^n$ whilst by comparing Theorem \ref{thm-general-poly} and Theorem \ref{thm-SU(n)-poly-negative}, we see that we have a rather big gap for the case of $SU(n)$.
			
			\item
			Varapolous showed that $A(\mathbb{T}, \rho_\alpha)$ is a $Q$-algebra if and only if $\alpha > 1/2$ \cite{Var72}. However, in general we can not expect $A(G, \om^\alpha_S)$ to be completely isomorphic to a $Q$-algebra since it may not be even completely isomorphic to a $Q$-space. Recall that an operator space $E$ is called a {\it $Q$-space} if it is a operator space quotient of a minimal operator space. More generally, the cb-distance of $E$ from a $Q$-space is defined by
				$$d_Q(E) = \inf \{ \norm{T}_{cb}\norm{T^{-1}}_{cb}\},$$
			where the infimum runs over all possible complete isomorphism $T : E \to F$ for some $Q$-space $F$. Clearly, $Q$-algebras are $Q$-spaces. Moreover, we have the following estimates (\cite[Proposition 5.4.16]{BLe04}).
				$$d_Q(C_n) = \sqrt{n}.$$
	Indeed, $A(SU(n), \om^\alpha_S)$ contains row Hilbert spaces of arbitrarily large dimensions so that $A(SU(n), \om^\alpha_S)$ is not completely isomorphic to a $Q$-space.
		\end{enumerate}
	}
	\end{rem}

\subsection{The case of compact connected non-simple Lie groups with dimension weights}\label{S:dim. weight-non-simple Lie}

In this section, we show that, for a non-simple compact Lie group $G$,
one can not find a Beurling--Fourier algebra on $G$, with dimension weights, which is isomorphic
to an operator algebra. Hence we need to restrict our attention to
simple cases (such as $SU(n)$) to obtain operator algebra in this case.

\begin{thm}\label{T:dim weight-non simple-non operator alg}
Let $G$ be a compact connected non-simple Lie group and $\alpha \geq 0$. Then
$A(G,\om_\alpha)$ is not isomorphic to an operator algebra.
\end{thm}

\begin{proof}
Since $G$ is not simple,
by \cite[6.5.6]{P}, $G\cong (P \times T)/A$, where $P$ is a product of compact connected simple Lie groups, $T$
 is an infinite compact connected abelian group and $A$ is a central subgroup of $P\times T$. This, in particular, implies that $G'\neq G$, where $G'$ is the derived subgroup of $G$. Hence $G/G'$ is an infinite compact connected abelian group. On the other hand, since $G'$ is compact, we can view $C(G/G')$ as a subalgebra of $C(G)$. With this identification, a straightforward computation shows that, for every
 $f\in C(G/G')\subset C(G)$ and $\pi \in \widehat{G}$,
 $$\hat{f}(\pi)=0 \ \ \text{if} \ \ d_\pi >1.$$ Thus
 \begin{eqnarray*}
 \|f\|_{A(G,\om_\alpha)} &= & \sum_{\pi \in \widehat{G}} d_\pi^{\alpha +1} \|\hat{f}(\pi)\|_1 \\
 &=& \sum_{d_\pi=1} d_\pi^{\alpha+1} \|\hat{f}(\pi)\|_1 \\
 &=& \sum_{\chi \in \widehat{G/G'}} |\hat{f}(\chi)|.
 \end{eqnarray*}
Thus the commutative group algebra $\ell^1(\widehat{G/G'})$ is a closed subalgebra
of $A(G, \om_\alpha)$. Hence if $A(G, \om_\alpha)$ is isomorphic to an operator algebra, then so is
$\ell^1(\widehat{G/G'})$. However this is impossible because $\ell^1(\widehat{G/G'})$ is not
Arens regular \cite{Young73}, and so, $A(G,\om_\alpha)$ is not isomorphic to an operator algebra.
\end{proof}

\subsection{The case of $SU(n)$ with dimension weights}

The case of dimension weights turns out to be more delicate than that of polynomial weights. We thus restrict our attention to $SU(n)$. The following estimate is crucial for the positive result, and it concerns with the relative dimensions in the irreducible decomposition of the tensor product of two irreducible representations of $SU(n)$, which is of an independent interest.
	\begin{cond}\label{con}\label{con1}
	Let $\pi_\lambda, \pi_\mu, \pi_\nu \in \widehat{SU(n)}$ with $\pi_\nu \subset \pi_\lambda \otimes \pi_\mu$. There is a constant $C(n)$ depending only on $n$ such that
		$$\frac{d_\nu}{d_\lambda d_\mu} \le C(n) \left(\frac{1}{\lambda_1+1} + \frac{1}{\mu_1+1} \right).$$
	\end{cond}
	\begin{thm}\label{T:Conj 1- n between 2 and 5}
	Condition \ref{con1} holds for all $n\ge 2$.
	\end{thm}
\begin{proof}
See \cite{CLS} for the proof of the general case. The proof for $n =3$ is may be acheived through much more elementary means and will be presented in the appendix A.
\end{proof}
\begin{rem}{\rm
In a preliminary version of this paper the authors presented a proof of Condition \ref{con1} for the case $3\le n \le 5$ using an elementary estimate. Very recently, Condition \ref{con1} has been settled down affirmatively \cite{CLS} for all $n\ge 2$. But the method of the proof uses a deep understanding about the connection between Gelfand-Tsetlin patterns and the combinatorics of Young tableaux, which requires techniques which are very different from the ones already used in this paper, even for the simplest nontrivial case of $n=3$. We believe that the elementary proof of Condition \ref{con1} for $n=3$ in the appendix A will be helpful to the interested readers.  In fact, the techniques in Appendix A may be extended to the cases of $n=4, 5$.  However, the the complexity of the estimates increases significantly and we do not know how to extend these techniques to cases of $n\ge 6$.
}
\end{rem}

Now we consider the case of $SU(n)$ with the dimension weights.

	\begin{thm}\label{thm-SU(n)-dim}
	Let $\om_\alpha$ be the dimension weight of order $\alpha$ on $\widehat{SU(n)}$, $n\ge 2$ (Definition \ref{def-dimension-weight}). Then $A(SU(n), \om_\alpha)$ is completely isomorphic to an operator algebra if $\alpha>\frac{d(SU(n))}{2}=\frac{n^2-1}{2}$ and is not completely isomorphic to an operator algebra if $\alpha\le \frac{1}{2}$.
	\end{thm}
	
\begin{proof}
For the positive direction, since by Theorem \ref{T:Conj 1- n between 2 and 5}, Condition 1 holds, we have
	\begin{align*}
	\frac{\om_\alpha(\pi_\nu)}{\om_\alpha(\pi_\lambda)\om_\alpha(\pi_\mu)}
	& \le (2C(n))^\alpha\left(\frac{1}{(\lambda_1 + 1)^\alpha} + \frac{1}{(\mu_1 + 1)^\alpha}\right)\\
	& = (2C(n))^\alpha\left(\frac{1}{(1+\norm{\pi_\lambda}_1)^\alpha} + \frac{1}{(1+\norm{\pi_\mu}_1)^\alpha}\right)
	\end{align*}
for any $\pi_\nu \subset \pi_\lambda\otimes \pi_\mu$.
The rest of the argument goes exactly the same as the one presented in the proof of Theorem \ref{thm-general-poly}.
Negative results follow from Theorem \ref{thm-restriction-dim}, Theorem \ref{thm-torus} and the fact that
a complete quotient of an operator algebra is again an operator algebra \cite[Proposition 2.3.4]{BLe04}.
\end{proof}

The following is an interesting consequence of the preceding theorem.
\begin{cor}\label{C:op. subalgebra of Fourier alg}
Let $k,n\in \mathbb{N}$ with $\frac{n^2-1}{2} < 2^k$, and let $G_k=SU(n)\times \cdots \times SU(n)$, $2^k$-times. Then
$A (SU(n), \om_{2^k})$ is a unital, infinite-dimensional closed subalgebra of the Fourier algebra $A(G_k)$
which is completely isomorphic to an operator algebra.
\end{cor}
\begin{proof}
A similar argument to the one given in the proof of \cite[Theorem 4.4]{LS} shows that
$A (SU(n), \om_{2^k})$ is a unital, infinite-dimensional closed subalgebra of
$A(G_k)$. Now the result is a direct concequence of Theorem \ref{thm-SU(n)-dim}.
\end{proof}

\subsection{The case of compact connected Lie groups with exponential weights}\label{S:Exp. weight-poly growth}

Let $G$ be a compact connected Lie group. For $0\leq \alpha \leq 1$, we recall the exponential weight $\gamma_\alpha = \gamma^\alpha_S$ in Definition \ref{def-poly-weight}.
In this section, we will study when the Beurling--Fourier algebra $A(G,\gamma_\alpha)$ is completely isomorphic to an operator algebra.
If we want to apply the same approach as before we need to find an appropriate decomposition of the function
	$$\frac{\gamma_\alpha(\sigma)}{\gamma_\alpha(\pi)\gamma_\alpha(\pi')}$$
for any $\pi, \pi', \sigma \in \widehat{G}$ with $\sigma \subset \pi\otimes \pi'$.
However, the lack of subadditivity of the function $e^{\tau_S(\pi)^\alpha}$ makes the problem more complicated.
Instead, we use the following estimate.

\begin{prop}\label{prop-exp-domination}
Let $0< \alpha < 1$ and $\beta \ge \max\{1, \frac{6}{\alpha(1-\alpha)}\}$. There is a constant $M$ depending only on $\alpha$, $\beta$, and the group $G$ such that
	$$\frac{\gamma_\alpha(\sigma)}{\gamma_\alpha(\pi)\gamma_\alpha(\pi')}
	\le M^2\frac{\om_\beta(\sigma)}{\om_\beta(\pi)\om_\beta(\pi')}$$
for any $\pi, \pi', \sigma \in \widehat{G}$ with $\sigma \subset \pi\otimes \pi'$.
\end{prop}
\begin{proof}
We present the proof in the appendix \ref{Appendix-exp-poly weight}.
\end{proof}

Now we have the results for exponential weights.
Note that in the case of exponential weights we have a better understanding of the negative results.

\begin{thm}\label{T:Expo weight-operator alg}
Let $G$ be a compact connected Lie group and $0< \alpha \leq 1$. Then:\\
$(i)$ The Beurling--Fourier algebra $A(G,\gamma_\alpha)$ is completely isomorphic to an operator algebra if $0<\alpha<1$.\\
$(ii)$ If $G$ is not simple, then $A(G,\gamma_1)$ is not isomorphic to an operator algebra.\\
$(iii)$ If $G=SU(n)$, then $A(G,\gamma_1)$ is not isomorphic to an operator algebra.
\end{thm}

\begin{proof}
First assume that $0<\alpha<1$ and take $\beta \geq \displaystyle \max\left \{1, \frac{6}{\alpha(1-\alpha)} \right \}$. Then,
by Proposition \ref{prop-exp-domination}, we have
	$$\frac{\gamma_\alpha(\sigma)}{\gamma_\alpha(\pi)\gamma_\alpha(\pi')}
	\le \frac{M^2 \om_\beta(\sigma)}{\om_\beta(\pi)\om_\beta(\pi')}
	\le M^2 \left(\frac{2}{(1+C)}\right)^\beta\left(\frac{1}{(1+\norm{\pi}_1)^\beta}
	+\frac{1}{(1+\norm{\pi'}_1)^\beta}\right).$$
If we take $\beta$ large enough, then by a similar argument to the one presented in the proof of Theorem \ref{thm-general-poly}, we can conclude that $A(G,\gamma_\alpha)$ is completely isomorphic to an operator algebra.
This proves (i).

For part (ii), similar to the proof of Theorem \ref{T:dim weight-non simple-non operator alg}, we can show that the commutative Beurling algebra
$\ell^1(\widehat{G/G'}, \tilde{\gamma_\alpha})$ is a closed subalgebra of $A(G, \gamma_1)$, where $\tilde{\gamma_\alpha} = (\gamma_\alpha)|_{\widehat{G/G'}}$.
Hence if $A(G, \gamma_1)$ is isomorphic to an operator algebra, then so is $\ell^1(\widehat{G/G'}, \tilde{\gamma_\alpha})$. However it follows routinely from the definition of $\tau_S$ (preceding to Definition \ref{def-poly-weight}) and the fact that $\widehat{G/G'}$ has a copy of $\z$ that $\ell^1(\z, e^{|\cdot|})$ is a closed subalgebra of $\ell^1(\widehat{G/G'}, \tilde{\gamma_\alpha})$. But $\ell^1(\z, e^{|\cdot|})$ is not Arens regular by \cite[Theorem 8.11]{DL}, and so, it can not be isomorphic to an operator algebra. Therefore $\ell^1(\widehat{G/G'}, \tilde{\gamma_\alpha})$ can not be isomorphic to an operator algebra. This completes the proof of (ii).

Finally, if $G=SU(n)$, then by Theorem \ref{thm-restriction-exp} we have that $\ell^1(\z, e^{|\cdot|})$ is a complete quotient of $A(SU(n), \gamma_1)$. Hence $A(SU(n), \gamma_1)$ is not isomorphic to an operator algebra since quotients of operator algebras are again operator algebras (see \cite[Proposition 2.3.4]{BLe04} for example).
\end{proof}

\begin{rem}
We note that when $\alpha =0$ we have $A(G,\gamma_0)=A(G)$, which is not Arens regular by \cite{F1}. Hence it can not be isomorphic to an operator algebra.
\end{rem}

\section*{Acknowledgements}
The authors would like to express their thanks to the referee for his/her valuable and kind comments, especially for pointing out Corollary \ref{cor-Varopoulos-extended} could be answered.

\appendix

\section{An elementary solution of Condition \ref{con1} for $n=3$}

Let $\pi_\lambda, \pi_\mu \in \widehat{SU(3)}$ with $\lambda = (\lambda_1, \lambda_2, \lambda_3 = 0)$ and $\mu = (\mu_1, \mu_2, \mu_3 = 0)$. Let $\pi_\nu \in \widehat{SU(3)}$ with $\nu = (\nu_1, \nu_2, \nu_3)$ and $\pi_\nu \subset \pi_\lambda \otimes \pi_\mu$.

We will introduce the following notations for simplicity.
	\begin{equation*}\label{eq-notation}
	\begin{cases}
	\lambda_{ij} & := \lambda_i - \lambda_j,\\
	\lambda\mu_i & := \lambda_i + \mu_i,\\
	\lambda\mu_{ij} & := \lambda_i - \lambda_j + \mu_i - \mu_j.
	\end{cases}
	\end{equation*}
The Littlewood-Richardson rule (\cite[(A.8)]{FH}) tells us that $\nu$ must be of the following form:
	$$\begin{cases}\nu_1 & = \lambda\mu_1 - \alpha_1 - \alpha_2 \\
	\nu_2 & = \lambda\mu_2 + \alpha_1 - \beta = \lambda\mu_2  + A \;\;\text{(where $A = \alpha_1 - \beta$)}\\
	\nu_3 & = \lambda\mu_3 + \alpha_2 + \beta = \alpha_2 + \beta,\end{cases}$$
where $\alpha_1$, $\alpha_2 \ge 0$ are the numbers of ``new" boxes with 1 in the second and the third row, respectively,
and $\beta \ge 0$ is the number of ``new" boxes with 2 in the third row.

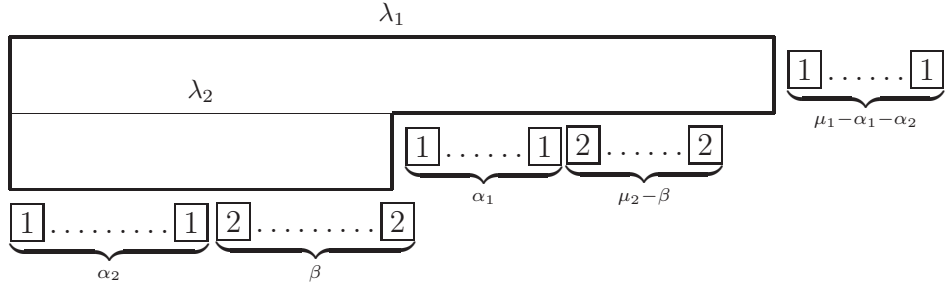
\setlength{\unitlength}{0.4in}
\begin{picture}(7,6)(0,0)
\linethickness{1pt}
\put (0,4){\line(1,0){10}}
\put (5,3){\line(1,0){5}}
\linethickness{0.2pt}
\put (0,3){\line(1,0){5}}
\linethickness{1pt}
\put(0,2){\line(1,0){5}}
\put(0,2){\line(0,1){2}}
\put(10,3){\line(0,1){1}}
\put(5,2){\line(0,1){1}}
\put(5,4.3){\makebox(0,0){$\lambda_1$}}
\put(2.5,3.3){\makebox(0,0){$\lambda_2$}}
\put(11.2,3.3){\makebox(0,0){$\Large{\underbrace{\fbox{1}\dots\dots \fbox{1}}_{\mu_1-\alpha_1-\alpha_2}}$}}
\put(6.2,2.3){\makebox(0,0){$\Large{\underbrace{\fbox{1}\dots\dots \fbox{1}}_{\alpha_1}}$}}
\put(8.3,2.3){\makebox(0,0){$\Large{\underbrace{\fbox{2}\dots\dots \fbox{2}}_{\mu_2-\beta}}$}}
\put(1.3,1.3){\makebox(0,0){$\Large{\underbrace{\fbox{1}\dots\dots\dots \fbox{1}}_{\alpha_2}}$}}
\put(4,1.3){\makebox(0,0){$\Large{\underbrace{\fbox{2}\dots\dots\dots \fbox{2}}_{\beta}}$}}
\end{picture}
\\
The Littlewood-Richardson rule also tells us that there are two kinds of restrictions on the parameters $\alpha_i$ and $\beta$.
The first one comes from ``Pieri's formula"(\cite[(A.7)]{FH}), which says that no two boxes in the same column can have the same number, so that we have
	\begin{align*}
	\alpha_1 & \le \lambda_{12}, & \alpha_2 + \beta & \le \lambda_{23} + \alpha_1\\
	\alpha_2 & \le \lambda_{23}.
	\end{align*}
The second one goes as follows.
When we list the new boxes from right to left, starting with the top row and working down, say from 1 to $k$-th boxes,
the number $i$ should appear no less than the number $i+1$, so that we have
	\begin{align*}
	\alpha_1 + \alpha_2 & \le \mu_{12} + \beta, & \beta & \le \mu_{23} \\
	\alpha_2& \le \mu_{12}.
	\end{align*}
Now we can extract constraints for $A = \alpha_1 - \beta$ as follows.
	\begin{equation}\label{eq-A}
	\begin{cases}A \le \min(\lambda_{12}, \mu_{12})\\ -A \le \min(\lambda_{23}, \mu_{23})\end{cases}.
	\end{equation}

Thus, we have to find an upper bound (independent of $\lambda$ and $\mu$) of
	\begin{align*}
	\frac{d_\nu}{d_\lambda d_\mu} & = \frac{(\nu_{12}+1)(\nu_{13} +2)(\nu_{23}+1)}{(\lambda_{12}+1)(\lambda_{13} +2)(\lambda_{23}+1)(\mu_{12}+1)(\mu_{13} +2)(\mu_{23}+1)}\\
	& = \frac{\nu_{13} +2}{(\lambda_{13} +2)(\mu_{13} +2)}\cdot \frac{(\nu_{12}+1)(\nu_{23}+1)}{(\lambda_{12}+1)(\lambda_{23}+1)(\mu_{12}+1)(\mu_{23}+1)} = I \cdot II.
	\end{align*}
For $I$ we have
	$$I = \frac{\lambda_1 + \mu_1 - \alpha_1 -2\alpha_2 -\beta +2}{(\lambda_1 +2)(\mu_1 +2)} \le \frac{\lambda_1 + \mu_1 +2}{(\lambda_1 +2)(\mu_1 +2)} \le 1.$$
For $II$ we have
	\begin{align*}
	II & = \frac{(\lambda\mu_{12} -\alpha_1 -\alpha_2 -A+1)(\lambda\mu_{23} + A -(\alpha_2 + \beta)+1)}{(\lambda_{12}+1)(\lambda_{23}+1)(\mu_{12}+1)(\mu_{23}+1)}\\
	& \le \frac{(\lambda\mu_{12} -A+1)(\lambda\mu_{23} + A+1)}{(\lambda_{12}+1)(\lambda_{23}+1)(\mu_{12}+1)(\mu_{23}+1)} = II',
	\end{align*}
where $A = \alpha_1 - \beta$.

Now we divide the cases into two parts, namely (1) $A\ge 0$, (2) $A<0$.
\vspace{0.2cm}

(1) When $A\ge 0$ by \eqref{eq-A} we have
	\begin{align*}
	II' & \le \frac{(\lambda\mu_{12}+1)(\lambda\mu_{23} + A+1)}{(\lambda_{12}+1)(\lambda_{23}+1)(\mu_{12}+1)(\mu_{23}+1)}\\
	& \le \frac{(\lambda\mu_{12}+1)(\lambda\mu_{23}+1)}{(\lambda_{12}+1)(\lambda_{23}+1)(\mu_{12}+1)(\mu_{23}+1)} + \frac{(\lambda\mu_{12}+1)\min(\lambda_{12}, \mu_{12})}{(\lambda_{12}+1)(\lambda_{23}+1)(\mu_{12}+1)(\mu_{23}+1)}\\
	& \le 1 + 2 = 3.
	\end{align*}
For the second term we used the following elementary inequality.
	$$\frac{(a+b+1)\min(a,b)}{(a+1)(b+1)} \le 2,\;\; a,b\ge 0.$$
(2) When $A<0$ we similarly have $II' \le 3$.
Note that $II''$ is symmetric via the correspondence
	$$(A, \lambda_{12}, \mu_{12}, \lambda_{23}, \mu_{23}) \longleftrightarrow (-A, \lambda_{23}, \mu_{23}, \lambda_{12}, \mu_{12}).$$

\section{The proof of Proposition \ref{prop-exp-domination}}\label{Appendix-exp-poly weight}

First, we recall the following lemma from \cite{LSS}.

\begin{lem}\label{L:estimate poly-expo weight}
Let $0<\alpha<1$ and take $\beta \geq \displaystyle \max\left\{1, \frac{6}{\alpha(1-\alpha)} \right\}$.
Define the functions $p: [0,\infty) \to \mathbb{R}$ and $q :  \mathbb{R}^+ \to  \mathbb{R}$ by
\begin{align}\label{Eq:estimate poly-expo weight}
p(x)=Cx^\alpha-\beta \ln (1+x) \ \ , \ \ q(x)=\frac{p(x)}{x}.
\end{align}
Then $p$ is increasing and $q$ is decreasing on $\displaystyle  \left[\left(\frac{\beta^2}{\alpha(1-\alpha)}\right)^{1/\alpha} , \infty \right)$.
\end{lem}

Now we present the second lemma.

\begin{lem}\label{L:Expo weight-bounded-Poly weight}
Let $0<\alpha<1$, $\beta \geq \displaystyle \max\left \{1, \frac{6}{\alpha(1-\alpha)} \right \}$,
and let $G$ be a compact group. Suppose that
$\tau : \widehat{G} \to [0,\infty)$ is a function satisfying
\begin{align}\label{Eq:lenght func-trai equality-double side}
\tau(\sigma) \leq \tau(\pi)+\tau(\pi').
\end{align}
for every $\pi, \pi' \in \widehat{G}$ and $\sigma \subset \pi\otimes \pi'$.
Let $p$ and $q$ be the functions defined in (\ref{Eq:estimate poly-expo weight})
and consider the function $\om : \widehat{G} \to [1,\infty)$ defined by
$$\om(\pi)=e^{p(\tau(\pi))}=e^{\tau(\pi)q(\tau(\pi))} \ \ \ (\pi\in \widehat{G}).$$
Then, for every $\pi, \pi' \in \widehat{G}$ and $\sigma \subset \pi\otimes \pi'$,
$$\om(\sg)\leq M^2\om(\pi)\om(\pi'),$$
where
\begin{align}\label{Eq:bound-poly expo weight}
M=\max \{ e^{p(t)-p(s)-p(r)}: t,s,r \in [0,2K]\cap \mathbb{Z} \}.
\end{align}
and
\begin{align}\label{Eq:increasing-decreasing interval}
K=\left(\frac{\beta^2}{\alpha(1-\alpha)}\right)^{1/\alpha}.
\end{align}
\end{lem}

\begin{proof}
By Lemma \ref{L:estimate poly-expo weight},
$p$ is increasing and $q$ is deceasing on $[K,\infty)$.
Let $\pi, \pi' \in \widehat{G}$ and $\sigma \subset \pi\otimes \pi'$ be a subrepresentation
of $\pi\otimes \pi'$.
We will prove the statement of the theorem by considering various
cases:\\
{\it Case I:} $\max\{ \tau(\pi), \tau(\pi')\}\leq K$. In this case, $\tau(\sg)\leq \tau(\pi)+\tau(\pi')\leq 2K$.
Hence
$$\frac{\om(\sg)}{\om(\pi)\om(\pi')}=e^{p(\tau(\sg))-p(\tau(\pi))-p(\tau(\pi'))} \leq M.$$
{\it Case II:} $\max\{ \tau(\pi), \tau(\pi')\}>  K$, $
\min\{ \tau(\pi), \tau(\pi')\}\leq  K$, and $\tau(\sg) < K$.
Without loss of generality, we can assume that $\tau(\pi)>K$ and $\tau(\pi')\leq K$.
Thus, by Lemma \ref{L:estimate poly-expo weight},
\begin{eqnarray*}
\om(\sg) & = & e^{p(\tau(\sg))} \\
&\leq & M e^{p(K)} \\
& \leq & M e^{p(\tau(\pi)+\tau(\pi'))} \\
&=& M e^{(\tau(\pi)+\tau(\pi'))q(\tau(\pi)+\tau(\pi'))}\\
&=& M e^{\tau(\pi)q(\tau(\pi)+\tau(\pi'))}e^{\tau(\pi')q(\tau(\pi)+\tau(\pi'))}\\
&\leq& M e^{\tau(\pi)q(\tau(\pi))}e^{Kq(K)}\\
&=& M \om(\pi)\om(\pi')e^{p(K)-p(\tau(\pi'))} \\
&\leq & M^2 \om(\pi)\om(\pi').
\end{eqnarray*}
{\it Case III:} $\max\{ \tau(\pi), \tau(\pi')\}>  K$,
$\min\{ \tau(\pi), \tau(\pi')\}\leq  K$, and $\tau(\sg) \geq K$.
Without loss of generality, we can assume that $\tau(\pi)>K$ and $\tau(\pi')\leq K$.
Then, by (\ref{Eq:lenght func-trai equality-double side}),
$$\tau(\pi)+\tau(\pi')\geq \tau(\sg)\geq K.$$
Thus, by Lemma \ref{L:estimate poly-expo weight},
\begin{eqnarray*}
\om(\sg) & = & e^{p(\tau(\sg))} \\
& \leq &  e^{p(\tau(\pi)+\tau(\pi'))} \\
&=&  e^{(\tau(\pi)+\tau(\pi'))q(\tau(\pi)+\tau(\pi'))}\\
&=&  e^{\tau(\pi)q(\tau(\pi)+\tau(\pi'))}e^{\tau(\pi')q(\tau(\pi')+\tau(\pi'))}\\
&\leq&  e^{\tau(\pi)q(\tau(\pi))}e^{Kq(K)}\\
&=&  \om(\pi)\om(\pi')e^{p(K)-p(\tau(\pi'))} \\
&\leq & M \om(\pi)\om(\pi').
\end{eqnarray*}
{\it Case IV:} $\min\{ \tau(\pi'), \tau(\pi)\}> K$ and $\tau(\sg)\leq K$.
In this case, we have
\begin{eqnarray*}
\om(\pi)\om(\pi') & = &  e^{p(\tau(\pi))+p(\tau(\pi'))} \\
& \geq & e^{2p(K)} \\
&=& e^{2p(K)-p(\tau(\sg))}\om(\sg) \\
&\geq & \frac{1}{M} \om(\sg).
\end{eqnarray*}
Hence
$$ \om(\sg)\leq M \om(\pi)\om(\pi').$$
{\it Case V:} $\min\{ \tau(\pi), \tau(\pi'), \tau(\sg)\}> K$.
In this case, by Lemma \ref{L:estimate poly-expo weight}, we have
\begin{eqnarray*}
\om(\sg) & = & e^{p(\tau(\sg))} \\
& \leq &  e^{p(\tau(\pi)+\tau(\pi'))} \\
&=& e^{(\tau(\pi)+\tau(y))q(\tau(\pi)+\tau(\pi'))}\\
&=& e^{\tau(\pi)q(\tau(\pi)+\tau(\pi'))}e^{\tau(\pi')q(\tau(\pi)+\tau(\pi'))}\\
&\leq& e^{\tau(\pi)q(\tau(\pi))}e^{\tau(\pi')q(\tau(\pi')}\\
&\leq& \om(\pi)\om(\pi').
\end{eqnarray*}
Therefore by comparing the above five cases and considering the fact that
$M\geq e^{-p(0)}=1$, it follows that
$$\om(\sg)\leq M^2 \om(\pi)\om(\pi').$$
\end{proof}
{\it Proof of Proposition \ref{prop-exp-domination}}:
If we set the function $\om$ by
	$$\om(\pi)=\frac{\gamma_\alpha(\pi)}{\om_\alpha(\pi)}
	=e^{\tau_S(\pi)^\alpha-\beta \ln (1+\tau_S(\pi))} \ \ (\pi\in \widehat{SU(n)}),$$
then by Lemma \ref{L:Expo weight-bounded-Poly weight}, for every $\pi, \pi' \in \widehat{SU(n)}$ and $\sigma \subset \pi\otimes \pi'$,
	$$\om(\sg)\leq M^2\om(\pi)\om(\pi'),$$
where $M$ is the constant defined in (\ref{Eq:bound-poly expo weight}).
This gives us the conclusion we wanted.

\end{document}